\documentclass[11pt,a4paper]{article}

\usepackage{latexsym,amsfonts,amsmath,amsthm,graphics}

\makeatletter
\providecommand*{\input@path}{}
\g@addto@macro\input@path{{../common/}}
\makeatother

\usepackage{graphicx}       

\usepackage[latin1]{inputenc}
\usepackage[T1]{fontenc}
\usepackage{microtype} 

\usepackage{cite}
\usepackage{array,colortbl}
\usepackage{amsmath,amsfonts,amssymb,bm,amsthm,mathtools,mathrsfs}
\DeclareFontFamily{U}{mathx}{\hyphenchar\font45}
\DeclareFontShape{U}{mathx}{m}{n}{
      <5> <6> <7> <8> <9> <10>
      <10.95> <12> <14.4> <17.28> <20.74> <24.88>
      mathx10
      }{}
\DeclareSymbolFont{mathx}{U}{mathx}{m}{n}
\DeclareFontSubstitution{U}{mathx}{m}{n}
\DeclareMathAccent{\widecheck}{0}{mathx}{"71}

\def\citep#1#2{\cite[{#1}]{#2}}

\newcommand{\RefAlg}[1]{Algorithm~\textup{\ref{#1}}}

\usepackage{algorithm}
\usepackage{algorithmic}

\newcommand{\bbZ}{{\mathbb{Z}}}
\newcommand{\N}{{\mathbb{N}}} 
\newcommand{\R}{{\mathbb{R}}} 
\newcommand{\Z}{{\mathbb{Z}}} 
\newcommand{\ZZ}{{\mathbb{Z}}} 

\DeclareSymbolFont{bbold}{U}{bbold}{m}{n}
\DeclareSymbolFontAlphabet{\mathbbold}{bbold}

\newcommand{\bsg}{{\boldsymbol{g}}}

\newcommand{\bsu}{{\boldsymbol{u}}}

\newcommand{\bsw}{{\boldsymbol{w}}}
\newcommand{\bsx}{{\boldsymbol{x}}}

\newcommand{\bsz}{{\boldsymbol{z}}}

\newcommand{\bsell}{{\boldsymbol{\ell}}}
\newcommand{\bszero}{{\boldsymbol{0}}} 

\newcommand{\bsgamma}{{\boldsymbol{\gamma}}}

\newcommand{\calF}{{\mathcal{F}}}

\newcommand{\calO}{{\mathcal{O}}}

\newcommand{\fraku}{{\mathfrak{u}}}

\newcommand{\setu}{{\mathfrak{u}}}

\newcommand{\rme}{{\mathrm{e}}}

\newcommand{\icomp}{\mathrm{i}}
\newcommand{\abs}[1]{\left\vert#1\right\vert}
\newcommand{\norm}[1]{\left\Vert#1\right\Vert}
\newcommand{\rd}{\,\mathrm{d}} 

\providecommand{\argmin}{\operatorname*{argmin}}

\usepackage[hidelinks,hypertexnames=false,hyperindex=true,pdfpagelabels,linktoc=all]{hyperref}
\usepackage{bookmark}
\makeatletter 
  \providecommand*{\toclevel@author}{999}
  \providecommand*{\toclevel@title}{0}
\makeatother

\usepackage{enumitem}

\theoremstyle{plain}
  \newtheorem{theorem}{Theorem}
  \newtheorem{proposition}{Proposition}
  \newtheorem{lemma}{Lemma}
  \newtheorem{corollary}{Corollary}
  
\theoremstyle{definition}
  \newtheorem{definition}{Definition}
  
\theoremstyle{remark}

\usepackage[T1]{fontenc}
\usepackage{lmodern}
\usepackage[a4paper]{geometry}
\usepackage{inputenc}
\usepackage{amsmath,amsfonts,amssymb}
\usepackage{bm}
\usepackage{mathtools}
\usepackage{hyperref}
\usepackage{tikz}
\usepackage{caption}
\usepackage{listings}
\usepackage{float}
\usepackage{pgfplots}
\usepackage{subcaption}
\usepackage{enumitem}
\usepackage{xcolor}

\usepackage[normalem]{ulem} 

\usepackage{booktabs}
\usepackage{multirow}

\allowdisplaybreaks

\usepackage{changebar}

\pgfplotsset{every tick label/.append style={font=\scriptsize}}
\newenvironment{customlegend}[1][]{
	\begingroup
	\csname pgfplots@init@cleared@structures\endcsname
	\pgfplotsset{#1}
}{
	\csname pgfplots@createlegend\endcsname
	\endgroup
}

\def\addlegendimage{\csname pgfplots@addlegendimage\endcsname}
\pgfplotsset{
	cycle list={%
		{draw=black,mark=star,solid},
		{draw=black, mark=square,solid},
		{draw=black,mark=+,solid},
		{black,mark=o},
		{draw=black, mark=none,solid}}
}

\pgfplotsset{every tick label/.append style={font=\scriptsize}}

\def\addlegendimage{\csname pgfplots@addlegendimage\endcsname}
\pgfplotsset{
	cycle list={%
		{draw=black,mark=star,solid},
		{draw=black, mark=square,solid},
		{draw=black,mark=+,solid},
		{black,mark=o},
		{draw=black, mark=none,solid}}
}

\definecolor{mycolor3}{rgb}{0.10000,0.80000,0.15000}

\definecolor{mycolor1-fig}{rgb}{0.00000,0.20000,0.50000}%
\definecolor{mycolor2-fig}{rgb}{0.00000,0.50000,0.00000}%
\definecolor{mycolor3-fig}{rgb}{0.85000,0.11800,0.09800}%

\newcommand{\tpmod}[1]{{\;(\operatorname{mod}\;#1)}}
\newcommand{\tmod}[1]{{\;({\rm{mod}}\; #1)}}
\newcommand{\supp}{\operatorname{supp}}

\usepackage{cancel}

\begin{document}
\title{On a reduced digit-by-digit component-by-component construction of lattice point sets}

\author{Peter Kritzer\thanks{P. Kritzer is supported by the Austrian Science Fund, Project F5506, 
which is part of the Special Research Program ``Quasi-Monte Carlo Methods: Theory and Applications'', and Project P34808. 
For the purpose of open access, the authors have applied a CC BY public copyright licence to any author accepted manuscript version arising from this submission.}, 
Onyekachi Osisiogu\thanks{O Osisiogu is partially supported by the Austrian Science Fund, Project F5506, 
which is part of the Special Research Program ``Quasi-Monte Carlo Methods: Theory and Applications''.}}

\date{\today}

\maketitle
\begin{abstract}
\noindent In this paper, we study an efficient algorithm for constructing point sets underlying quasi-Monte Carlo integration rules 
for weighted Korobov classes. The algorithm presented is a reduced fast component-by-component digit-by-digit (CBC-DBD) algorithm, 
which useful for to situations where the weights in the function space show a sufficiently fast decay. 
The advantage of the algorithm presented here is that the computational effort can be independent of the dimension 
of the integration problem to be treated if suitable assumptions on the integrand are met.
The new reduced CBC-DBD algorithm is designed to work for the construction of lattice point sets, 
and the corresponding integration rules (so-called lattice rules) can be used to treat functions in different 
kinds of function spaces. We show that the integration rules constructed by our algorithm satisfy error 
bounds of almost optimal convergence order. Furthermore, we give details on an efficient implementation such that we obtain a considerable 
speed-up of a previously known CBC-DBD algorithm that has been studied in \cite{EKNO21}. 
This improvement is illustrated by numerical results.
\end{abstract}



\section{Introduction}\label{sec:intro}
In  this  paper  we  study  numerical  integration  of functions  defined  over  the $d$-dimensional unit cube. 
For an integrable function $f: [0,1)^d\rightarrow \R$, we denote the integral of $f$ by  
\[ 
I_d(f) := \int_{[0,1]^d} f(\bsx) \rd\bsx,
\] 
and we study the efficient construction of high-dimensional \emph{quadrature rules}  
of the form
\begin{equation*} 
Q_N(f, (\bsx_k)_{k=0}^{N-1}) :=	\frac{1}{N}\sum_{k=0}^{N-1}  \, f(\bsx_k) 
\end{equation*} 
for numerically approximating $I_d$. We assume that the integrand $f$ lies in a Banach space $(\calF,\norm{\cdot}_{\calF})$, 
and that the integration nodes $\bsx_0,\ldots,\bsx_{N-1}\in [0,1)^d$ are to be chosen deterministically.  
Here $Q_N$ is called a \emph{quasi-Monte Carlo (QMC)} rule, which is a special case of an equal-weight integration rule.
One way to measure the quality of a QMC integration rule $Q_N$ is to consider the \emph{worst-case error} in the unit ball of the 
space $(\calF,\norm{\cdot}_{\calF}),$ i.e.,
\[	
e_{N,d}(Q_N,\calF) :=	\sup_{\substack{f \in \calF \\ \|f\|_{\calF} \le 1}} \left|I_d(f) - Q_N(f, (\bsx_k)_{k=0}^{N-1}) \right|. 
\]
In general, it is highly non-trivial to choose the set of integration nodes such that the resulting rule has a low worst-case error. 
There are two main types of QMC point sets $\mathcal{P} = \{\bsx_0, \bsx_1, \dots, \bsx_{N-1}\}$, 
namely \emph{lattice point sets} and \emph{digital nets}. In this paper we consider lattice point sets, 
which are obtained by considering discrete subsets of $\R^d$ that are closed under both addition and subtraction 
and contain $\Z^d$ as a subset; by intersecting such a discrete set with the $d$-dimensional 
unit cube $[0,1)^d$, we then obtain a lattice point set. We are interested in a special kind of 
lattice point sets that is essentially based on one \emph{generating vector} $\bsz$. 
These lattice point sets were first introduced by Korobov \cite{Kor59} and independently by Hlawka \cite{H62} 
and are among the most prominently studied QMC point sets to approximate multivariate integrals 
(see standard textbooks such as \cite{DKP22, HW81, N92b, SJ94}, and also \cite{DKS13}). 
For a natural number $N \in \N$ and an integer generating vector $\bsz \in \{1,2,\dots,N-1\}^d$, a point set with points of the form
\[  
 \bsx_k := \left\lbrace \frac{k}{N}\bsz\right\rbrace = \left( \left\{ \frac{k z_1}{N}\right\} , \ldots , 
 \left\{ \frac{k z_d}{N}\right\} \right)\in [0, 1)^d,\qquad\text{for } k = 0, 1, \ldots, N-1,
\] 
is called a \emph{rank-1 lattice point set}, which we shall sometimes denote by $P(\bsz,N)$. 
For vectors $\bsx \in\R^d$ we apply $\{\cdot\}$ component-wise, 
where $\{x\}=x-\lfloor x \rfloor$ denotes the fractional part of $x$. A QMC rule using such a point set as integration nodes 
is called a \emph{(rank-1) lattice rule} with generating vector $\bsz$. We remark that, given $N$ and $d$, 
a rank-1 lattice rule is completely determined by the choice of the generating vector $\bsz=(z_1,\ldots,z_d) \in \bbZ_N^d$, 
where we write $\bbZ_N := \{0, \ldots, N-1\}$. However, it should be obvious that not every choice of a generating vector $\bsz$ 
also yields a rank-1 lattice rule with good quality for approximating the integral, and it is usually necessary to 
tailor the choice of the integration nodes to the function space $\calF$ under consideration. For lattice rules we consider 
certain Banach spaces which are based on assuming sufficient decay of the Fourier coefficients of 
their elements to guarantee certain smoothness properties. These spaces are called weighted Korobov classes, 
which we will denote by $E_{d,\bsgamma}^{\alpha}$, where $d$ denotes the number of variables the functions depend on, 
$\alpha>1$ is a real number frequently referred to as the smoothness parameter, and $\bsgamma = (\gamma_j)_{j\ge 1}$ 
is a sequence of strictly positive reals, which model the importance of different coordinates. 
The idea of these additional parameters $\gamma_j$, which we refer to as weights in the definition of the function spaces 
under consideration, goes back to Sloan and Wo\'{z}niakowski \cite{SW98}, see also \cite{SW01,Hic98}, 
and this will be made more precise by incorporating the weights in the norm of the space $E_{d,\bsgamma}^{\alpha}$ in Section~\ref{sec:space}. 
If we denote the variables in the integration problem by $\bsx=(x_1,\ldots,x_d)$,
a small value of $\gamma_{j}$ corresponds to a low influence of the variable $x_j$, 
while a large $\gamma_{j}$ means high influence of $x_j$.  
Regarding the role of these weights in integration problems, we favor a situation in which the weights decay
sufficiently fast for coordinates with increasing indices, which helps in vanquishing the curse of dimensionality 
that is inherent to many high-dimensional problems. Indeed, under certain summability conditions on the weights, 
it is even possible to obtain bounds on the integration error that do not depend on the dimension of the problem at all. 
Such a situation is called \emph{strong polynomial tractability}, see, e.g., \cite{NW08}, for a general reference.

Regarding the construction of generating vectors, for dimensions $d\le 2$ explicit constructions of good generating vectors are available, 
see, e.g., \cite{SJ94,N92b}, but there are no explicit constructions of good generating vectors known for $d > 2$; 
how to find a generating vector $\bsz$ that guarantees a low worst-case error of integration in a given function space is 
a crucial question regarding rank-1 lattice rules. One way to find good generating vectors is the component-by-component (CBC) 
construction, which is based on a greedy algorithm choosing one component of the generating vector at a time. 
It was shown in \cite{K03} for prime $N$ (see also \cite{D04} for the case of composite $N$)
that it is possible to find generating vectors yielding essentially optimal results for certain spaces of $d$-variate 
functions by the CBC construction. Furthermore, it was shown in \cite{NC06b} that a fast CBC construction 
can reduce the computational cost of these algorithms to only $\mathcal{O}(dN\log N)$, by making use 
of a clever ordering of the points of a rank-1 lattice, and the Fast Fourier Transform.

In the paper \cite{DKLP15} the weights $\bsgamma=(\gamma_j)_{j\ge 1}$ in a given function space 
were incorporated in the CBC construction of lattice rules. 
To be more precise, depending on the weight $\gamma_j$, $j\in\N$, 
the search space for the $j$th component $z_j$ of the generating vector $\bsz$ 
can be shrinked as compared to the usual CBC construction,  
and this reduction is the motivation for the modified CBC algorithm to be called the 
``reduced'' CBC construction. Under suitable assumptions on the integrands considered, it 
was shown in \cite{DKLP15} that the lattice point sets obtained by the reduced CBC construction can 
still yield excellent results when used as integration nodes in a QMC rule, while the 
construction cost can be made independent of the dimension $d$ of the integration problem.
It was also shown in \cite{DKLP15} that the fast CBC construction principle of Nuyens and Cools, 
which was mentioned above, can be adapted to the reduced CBC construction.

A different construction algorithm for generating vectors of good lattice rules inspired by articles of Korobov (see \cite{Kor63} and \cite{Kor82}, 
and \cite{Kor82Eng} for an English translation) was dealt with in the recent paper \cite{EKNO21}, where a so-called component-by-component digit-by-digit 
(CBC-DBD) algorithm was presented. In this approach, one constructs the generating vector $\bsz$ in a component-by-component (CBC) fashion, in which each component $z_j$ 
is assembled digit-by-digit (DBD). That is, for a number  $N= 2^m$ of points we greedily construct the components $z_j$ bit-by-bit starting from the least significant bit. 
It was shown that the rules constructed by the CBC-DBD algorithm yield a convergence rate that is arbitrarily close to the optimal rate. 
It is also possible to have a fast implementation of the CBC-DBD algorithm which has a computational cost of $\mathcal{O}(dN \log N)$, 
and also numerical results on the performance of the CBC-DBD algorithm are presented in \cite{EKNO21}. We remark that the error analysis for the CBC-DBD algorithm 
is such that no prior knowledge of the smoothness parameter $\alpha$ is required to construct the generating vector. 
The resulting generating vector will still deliver the near optimal rate of convergence, for arbitrary smoothness parameters $\alpha > 1$, 
and this result can be stated independently of the dimension, assuming that the weights satisfy certain summability conditions. 

In the present paper, we would like to study a combination of the reduced CBC algorithm from \cite{DKLP15} 
and the CBC-DBD construction algorithm approach from \cite{EKNO21}, and present a \emph{reduced (fast) CBC-DBD algorithm} 
for generating vectors of good lattice rules with a large number of lattice points in high dimensions, assuming 
sufficiently fast decaying weights. This new algorithm will work by constructing the generating vector $\bsz$ 
in a component-by-component (CBC) fashion in which each component $z_j$ is assembled digit-by-digit (DBD). 
However, the search space for each component will be reduced in comparison to what is shown in \cite{EKNO21}, 
thus speeding up the construction method. We will show that for suitable choices of weights the construction cost of 
the reduced fast CBC-DBD algorithm can be made independent of the dimension. 
The main aim of this paper is to show that one can reduce the construction cost 
of the lattice rule by making the search for later components smaller than the earlier ones, while still achieving (strong) polynomial tractability. 

The structure of this paper is as follows. In the subsequent section (Section \ref{sec:space}), 
we outline the general setting, and then our construction algorithm as well as the main result 
are stated and proved in Section \ref{sec:construction}. Section \ref{sec:fast_impl} contains remarks on how to efficiently implement our 
newly found algorithm and some numerical results.

Regarding notation, let $\N :=\{1,2,\dots\}$ be the set of natural numbers, let $\N_0 := \{0,1,2,\dots\}$, 
and let $\Z$ be the set of integers. Additionally, let the set of non-zero integers be denoted by $\Z_* :=\Z\setminus \{0\}$. 
For sets of components we use fraktur font, e.g., $\setu \subset \N$. To denote the projection of a vector $\bsx \in [0,1)^d$ 
or $\bsell \in \Z^d$ to the components in a set $\setu \subseteq [d]:=\{1,\dots,d\}$, we write 
$\bsx_\setu := (x_j)_{j\in\setu}$ or $\bsell_\setu := (\ell_j)_{j\in\setu}$, respectively.


\section{Setting and overview}\label{sec:space}
We consider one-periodic real-valued $L_2$ functions defined on $[0,1]^d$ with absolutely convergent Fourier series
\begin{align*}
	f(\bsx)
	=
	\sum_{\bsell \in \Z^d} \hat{f}(\bsell) \, \rme^{2 \pi \icomp \bsell \cdot \bsx}
	\qquad \text{with} \qquad
	\hat{f}(\bsell)
	:=
	\int_{[0,1]^d} f(\bsx) \, \rme^{-2 \pi \icomp \bsell \cdot \bsx} \rd \bsx ,
\end{align*}
where $\hat{f}(\bsell)$ are the Fourier coefficients of $f$ and 
$\bsell \cdot \bsx := \sum_{j=1}^d \ell_j x_j = \ell_1x_1 + \cdots + \ell_dx_d$ is the vector dot product. 

It is known that rank-1 lattice point sets, as introduced in Section \ref{sec:intro}, have a 
property commonly referred to as the \emph{character property}. Indeed, for a rank-1 lattice point set 
with generating vector $\bsz\in\Z^d$, it is true that 
\[
 \frac1N \sum_{k=0}^{N-1} \rme^{2\pi\icomp (\bsell\cdot\bsz) \, k / N}
 =\begin{cases}    
  1 & \text{if } \bsell\cdot\bsz \equiv 0 \tmod{N}, \\ 
  0 & \text{otherwise.} 
\end{cases}
\]

We introduce the dual of a rank-1 lattice, which is of great importance in representing 
the error of approximating an integral by a lattice rule $Q_N$.

\begin{definition}[Dual lattice] 
Let $N\geq 2$  be  the  number  of  points and let $\bsz \in\{1,\dots,N-1\}^d$
be the generating vector of a rank-1 lattice point set $P(\bsz,N)$.
The dual lattice $\Lambda^\top (\bsz,N)$ of a rank-1 lattice $P(\bsz,N)$ is given by
\begin{align*}
\mathcal{D} = \Lambda^\top(\bsz,N) := \{\bsell \in \Z^d \ | \ \bsell \cdot \bsz \equiv 0 \tpmod{N} \}.
\end{align*}
For a non-empty set $\setu \subseteq [d],$ we can further define
\[
\mathcal{D}_\setu = \mathcal{D}_\setu(\bsz_\setu) := \{\bsell_\setu \in \Z^{|\setu|}_* \ | \ \bsell_\setu \cdot \bsz_\setu \equiv 0 \tpmod{N} \}.
\]
\end{definition}
Then, to obtain the integration error for a given function $f$ in terms of its Fourier coefficients, we interchange 
the order of summation and use the character property of the lattice points, and get
\begin{eqnarray*}
  Q_N(f, \bsz) - I_d(f)  
  &= &\sum_{\bszero \ne \bsell \in \bbZ^d} \hat{f}(\bsell) \left[ \frac1N \sum_{k=0}^{N-1} \rme^{2\pi\icomp (\bsell\cdot\bsz) \, k / N} \right] 
  \nonumber\\   
  &=&\!\!\!\!\sum_{\substack{\bszero \ne \bsell \in \Z^d\\\bsell\cdot \bsz\equiv 0 \tmod{N}}} \!\!\!\!\hat{f}(\bsell) \nonumber\\ 
  &=& \sum_{\bszero \ne \bsell \in \bbZ^d} \hat{f}(\bsell) \, \delta_N(\bsell \cdot \bsz),
\end{eqnarray*} 
where, for $a \in \Z$,
\[
\delta_N(a):=  \frac1N \sum_{k=0}^{N-1} \rme^{2\pi\icomp a k / N}=  
\begin{cases}    
  1 & \text{if } a \equiv 0 \tmod{N}, \\ 
  0 & \text{otherwise.} 
\end{cases}
\]

\noindent For a vector $\bsell \in \Z^d$, smoothness parameter $\alpha >1$, and strictly positive weights 
$\bsgamma=(\gamma_j)_{j\ge 1}$, we define
\begin{equation}\label{eq:decay_func} 
	r_{\alpha,\bsgamma}(\bsell)
	:=
	\prod_{j\in\supp(\bsell)} \gamma_j \abs{\ell_j}^\alpha,
\end{equation} 
where $\supp(\bsell) := \{ j \in [d]: \ell_j \ne 0 \}$ is the support of $\bsell$. 
As usual, we define the empty product as one such
that $r_{\alpha,\bsgamma}(\bszero) = 1$. Formally, we will also use a function $r_{1,\bsgamma}$ below, 
which is obtained by replacing $\alpha$ by 1 in \eqref{eq:decay_func}.

We then define our function space $E_{d,\bsgamma}^\alpha$ of one-periodic real-valued $L_2$ functions defined on $[0,1]^d$ 
with absolutely convergent Fourier series. The norm in this space is given as
\begin{align*}
	\|f\|_{E_{d,\bsgamma}^{\alpha}} 
	:= 
	\sup_{\bsell \in \Z^d} |\hat{f}(\bsell)| \, r_{\alpha,\bsgamma}(\bsell),
\end{align*}
and for $\alpha > 1$, dimension $d\in\N$, and positive weight sequence 
$\bsgamma = (\gamma_j)_{j\ge 1}$, our weighted function space is
\begin{equation}\label{eq:f_space}
	E_{d,\bsgamma}^{\alpha}
	:=
	\left\{f \in L^2([0,1]^d) \colon \|f\|_{E_{d,\bsgamma}^{\alpha}}  < \infty \right\}
	.
\end{equation}
It is known (see, for example, \cite{DKP22}) that the worst-case error of an
$N$-point rank-1 lattice rule generated by $\bsz \in \Z^d$ in the space $E_{d,\bsgamma}^\alpha$ is given by
\begin{equation}\label{eq:wce-infty}
		e_{N,d,\alpha,\bsgamma}(\bsz)
		=
		\sum_{\substack{\bszero \ne \bsell \in \mathcal{D}}} r_{\alpha,\bsgamma}^{-1}(\bsell) 
		=
		\sum_{\bszero \ne \bsell \in \Z^d} \frac{\delta_N(\bsell \cdot \bsz)}{r_{\alpha,\bsgamma}(\bsell)}
		.
	\end{equation}
We remark that the worst-case error \eqref{eq:wce-infty} is sometimes referred to as $P_\alpha$ in the literature on 
lattice rules. As pointed out above, it is the main goal of the present paper to state a new effective construction method 
for good instances of $\bsz$ such that $e_{N,d,\alpha,\bsgamma}(\bsz)$ is reasonably small.

It is well known, see, e.g., \cite{SJ94}, that the optimal convergence rate of the worst-case integration error in the space 
$E_{d,\bsgamma}^\alpha$ is of order $\calO(N^{-\alpha})$. Furthermore, we would like to remark that in the literature 
on lattice rules often a slightly different, but related, function space than $E_{d,\bsgamma}^\alpha$ is studied. Indeed, 
by modifying the norm in the space to $\left(\sum_{\bsell\in\Z^d} |\hat{f}(\bsell)|^2 \, r_{\alpha,\bsgamma}(\bsell)\right)^{1/2}$, 
one obtains a reproducing kernel Hilbert space, which is often referred to as ``Korobov space''. It is known (see, e.g., \cite{DKS13})
that the worst-case error in that space corresponds to the square root of the worst-case error in $E_{d,\bsgamma}^\alpha$, 
and so all results shown in the present paper can directly be carried over to related results for the Korobov space. 
Note also that, by using embedding results, any results on the error in $E_{s,\bsgamma}^\alpha$ 
immediately yield results on the worst-case error of ``tent-transformed'' lattice rules in certain Sobolev 
spaces of functions whose mixed partial derivatives of order~$1$ and~$2$ in each variable are square integrable 
(with $\alpha=2$ and $\alpha=4$ in our notation, respectively), see \cite{CKNS16,DKP22,DNP14,GSY19} for further details.

\subsection{The quality criterion used in this paper} 

In this section and the following we shall always assume that the number $N$ of points equals $2^m$, where $m\in\N$. 
Furthermore, we assume that we are given a sequence of non-negative integers $(w_j)_{j\ge 1}$ ordered in a non-decreasing fashion, 
to be more precise, $w_1,w_2,\ldots\in\N_0$ with $0= w_1 \le w_2 \le \cdots  $, and we put $Y_j:= 2^{w_j}$ for $j \in [d]$. 
The numbers $w_j$ for $j\ge 1$ are also called \emph{reduction indices}, as they will help to reduce the 
computational cost of the CBC-DBD algorithm studied in this paper.
In what follows, we set $d^*$ as the largest $j$ such that $w_j <m$, i.e., $d^*:= \max\{j\in \N: w_j <m\}$. 
Moreover, we will usually (unless stated otherwise) assume that the generating vector $\bsz$ of a 
lattice rule under consideration takes the form $\bsz = (Y_1z_1, \dots, Y_dz_d) \in \Z^d$, 
where $z_j \in \{1,3,5,\dots, 2^{m-w_j}-1\}$ for $j \in [d^*]$, and $z_j$ is odd if $j>d^*$ (in fact, we 
will often choose $z_j=1$ for $j>d^*$, but any choice of odd numbers $z_j$ for those $j$ would yield the same results for 
our lattices rules).

We define the following quality measure $T_\bsgamma(N,\bsz,\bsw)$ for $N \in \N$ and $\bsw = (w_1,\dots,w_d) \in \N_0^d$ 
with $0=w_1\le w_2\le \cdots$, by
\begin{equation}\label{eq:qual_meas}
	T_\gamma(N,\bsz,\bsw) 
	:= 
	\sum_{\bszero \ne \bsell \in M_{N,d,\bsw}} \frac{\delta_N(\bsell \cdot \bsz)}{r_{1,\bsgamma}(\bsell)}, 
\end{equation}      
where the set $M_{N,d,\bsw}$ is defined as follows. Let
\begin{equation*}
	M_{N,j,w_j}
	:=
	\{-(2^{\max(0,m-w_j)}-1),\ldots,2^{\max(0,m-w_j)}-1\}\qquad \text{for } j\in [d].
\end{equation*}
Furthermore, we write
\begin{equation*}
	\{M_{N,j,w_j}^*
	:=
	\left(\{-(2^{\max(0,m-w_j)}-1),\ldots,2^{\max(0,m-w_j)}-1\} \setminus \{0\}\right)
	.
\end{equation*}Note that 
$M_{N,j,w_j}^*=\emptyset$ for $j>d^*$. Furthermore, let
\[
 M_{N,d,\bsw}=M_{N,1,w_1}\times\cdots\times M_{N,d,w_d},
\qquad
\text{and}
\qquad
 M_{N,d,\bsw}^*=M_{N,1,w_1}^*\times\cdots\times M_{N,d,w_d}^*.
\]
Then, for $\emptyset\neq\fraku\subseteq [d]$, 
by $M_{N,\abs{\fraku},\bsw_{\fraku}}$ and $M_{N,\abs{\fraku},\bsw_{\fraku}}^*$, we denote the 
projections of $M_{N,d,\bsw}$ and $M_{N,d,\bsw}^*$ onto the components with indices 
in $\fraku$, respectively.
Likewise we define the auxiliary quantity
\begin{equation*}
	T_{\alpha,\bsgamma}(N,\bsz,\bsw) 
	:= 
	\sum_{\bszero \ne \bsell \in M_{N,d,\bsw}} \frac{\delta_N(\bsell \cdot \bsz)}{r_{\alpha,\bsgamma}(\bsell)}
	.
\end{equation*}
Let us denote the Riemann zeta function by $\zeta(\alpha) := \sum_{n=1}^{\infty} n^{-\alpha}$ for $\alpha>1$.
Moreover, we write, for a non-empty set $\setu\subseteq [d]$, 
\[
 \gamma_{\setu}:=\prod_{j\in\setu} \gamma_j,
\]
and we put $\gamma_\emptyset =1$.

The following proposition shows that we can use $T_{\alpha,\bsgamma}(N,\bsz,\bsw)$ as a suitable approximation to 
the worst-case error of a rank-1 lattice rule. 
\begin{proposition}\label{prop:trunc_error}
	Let $N=2^m$, let $\bsgamma = (\gamma_j)_{j\ge 1}$ be a sequence of positive weights, let 
	$\bsz = (Y_1z_1,\ldots,Y_dz_d) \in \Z^d$, and let $\bsw = (w_j)_{j\ge1}$ be a sequence in $\N_0$, 
	with $0=w_1\le w_2 \le \cdots$. Furthermore, assume that all $z_j$, $j\in [d]$, are odd, and that 
	$z_j \in \{1,3,5,\dots, 2^{m-w_j}-1\}$ for $j \in [d^*]$.
	Then, for $\alpha>1$, we have that
	\begin{eqnarray*}
	    e_{N,d,\alpha,\bsgamma}(\bsz) - T_{\alpha,\bsgamma}(N,\bsz,\bsw)
	    &=&
		\sum_{\bszero \ne \bsell \in \Z^d} \frac{\delta_N(\bsell \cdot \bsz)}{r_{\alpha,\bsgamma}(\bsell)}
		-
		\sum_{\bszero \ne \bsell \in M_{N,d,\bsw}} \frac{\delta_N(\bsell \cdot \bsz)}{r_{\alpha,\bsgamma}(\bsell)}\\ 
	    &\le&
		\sum_{\emptyset\neq \setu \subseteq [d]} \gamma_\setu \, \frac{(4\zeta (\alpha))^{\abs{\setu}}}{2^{\alpha\max(0,m-\max_{j\in\setu}w_j)}}
		.
	\end{eqnarray*}
\end{proposition}

\begin{proof}
We can rewrite the difference $e_{N,d,\alpha,\bsgamma}(\bsz) - T_{\alpha,\bsgamma}(N,\bsz,\bsw)$ as 
\[
 \sum_{\emptyset\neq \fraku \subseteq [d]} 
 \left(\sum_{\bsell_{\fraku}\in\ZZ_*^{\abs{\fraku}}}
 \frac{\delta_N (\bsell_{\fraku}\cdot\bsz_{\fraku})}{r_{\alpha,\bsgamma_\setu} (\bsell_{\fraku})}
 -\sum_{\bsell_{\fraku}\in M_{N,|\fraku|,\bsw_{\fraku}}^*}
 \frac{\delta_N (\bsell_{\fraku}\cdot\bsz_{\fraku})}{r_{\alpha,\bsgamma_\setu} (\bsell_{\fraku})}\right),
\]
motivating us to define, for $\fraku\neq \emptyset$, $\fraku\subseteq [d]$,
\[
 T_{\fraku,\bsw_{\fraku}}:=
 \sum_{\bsell_{\fraku}\in\ZZ_*^{\abs{\fraku}}}
 \frac{\delta_N (\bsell_{\fraku}\cdot\bsz_{\fraku})}{r_{\alpha,\gamma_\setu} (\bsell_{\fraku})}
 -\sum_{\bsell_{\fraku}\in M_{N,|\fraku|,\bsw_{\fraku}}^*}
 \frac{\delta_N (\bsell_{\fraku}\cdot\bsz_{\fraku})}{r_{\alpha,\gamma_\fraku} (\bsell_{\fraku})}.
\]

In the following we distinguish two cases for $\emptyset\neq\fraku\subseteq [d]$, depending on whether $\abs{\fraku}=1$ 

\medskip

\textbf{Case 1:} Suppose that $\abs{\fraku}=1$, i.e., $\fraku=\{j\}$ for some $j\in[d]$. Then we have
\begin{eqnarray*}
T_{\fraku,\bsw_{\fraku}}
&=& T_{\{j\},w_j}\\
&=& \sum_{\ell_j \in \ZZ_*} \frac{\delta_N (\ell_j Y_j z_j)}{r_{\alpha,\gamma_{\{j\}}} (\ell_j)}
- \sum_{\ell_j \in M_{N,\{j\},w_j}^*} \frac{\delta_N (\ell_j Y_j z_j)}{r_{\alpha,\gamma_{\{j\}}} (\ell_j)}\\
&=& \sum_{\abs{\ell_j}\ge 2^{\max (0,m-w_j)}} \frac{\delta_N (\ell_j Y_j z_j)}{r_{\alpha,\gamma_{\{j\}}} (\ell_j)}\\
&=& \sum_{\abs{\ell_j}\ge 2^{\max (0,m-w_j)}} \gamma_{j} \frac{\delta_N (\ell_j Y_j z_j)}{\abs{\ell_j}^\alpha}\\ 
&=& 2\gamma_{j} \sum_{t=1}^\infty \frac{1}{(t2^{\max (0,m-w_j)})^\alpha}=\frac{2\zeta (\alpha)}{2^{\alpha \max (0,m-w_j)}}\gamma_{j},
\end{eqnarray*} 
which follows from the fact that if $\ell_j Y_j z_j \equiv 0 \text{ (mod $N$)}$, this is equivalent to 
$\ell_j \equiv 0 \text{ (mod $2^{\max (0,m-w_j)}$})$.

\medskip

\textbf{Case 2:} 
For $\fraku\neq \emptyset$, $\fraku\subseteq [d]$ with $\abs{\fraku}>1$, and $i\in\fraku$, 
we write, for short, $\bsell_{\fraku\setminus \{i\}}, \bsz_{\fraku\setminus\{i\}}\in\ZZ^{\abs{\fraku}-1}$ to denote 
the projections of $\bsell$ and $\bsz$, respectively, onto those components with indices in $\fraku\setminus \{i\}$.

In this case, we estimate 
\[
 T_{\fraku,\bsw_{\fraku}} \le \sum_{i\in\fraku} \sum_{\bsell_{\fraku \setminus \{i\}}\in\ZZ_*^{\abs{\fraku}-1}} 
 \sum_{\abs{\ell_i}\ge 2^{\max (0,m-w_i)}} 
 \frac{\delta_N (\ell_i Y_i z_i + \bsell_{\fraku \setminus \{i\}}\cdot\bsz_{\fraku \setminus \{i\}})}{r_{\alpha,\gamma_{\fraku}}(\bsell_{\fraku})}.
\]
For $i\in\fraku$ and $\bsell_{\fraku \setminus \{i\}}\in\ZZ_*^{\abs{\fraku}-1}$, we write 
$b:= \bsell_{\fraku \setminus \{i\}}\cdot\bsz_{\fraku \setminus \{i\}}$, and consider the expression
\begin{eqnarray*}
 \lefteqn{\sum_{\abs{\ell_i}\ge 2^{\max (0,m-w_i)}} 
 \frac{\delta_N (\ell_i Y_i z_i + b)}{r_{\alpha,\gamma_{\fraku}}(\bsell_{\fraku})}
 =
 \gamma_{\fraku} \sum_{\abs{\ell_i}\ge 2^{\max (0,m-w_i)}} 
 \frac{\delta_N (\ell_i Y_i z_i + b)}{\prod_{j\in\fraku} \abs{\ell_j}^\alpha}}\\
 &=&\gamma_{\fraku} \left(\prod_{\substack{j\in\fraku\\ j\neq i}} \abs{\ell_j}^{-\alpha}\right) 
 \sum_{\abs{\ell_i}\ge 2^{\max (0,m-w_i)}} 
 \frac{\delta_N (\ell_i Y_i z_i + b)}{\abs{\ell_i}^\alpha}\\
 &=& \gamma_{\fraku} \left(\prod_{\substack{j\in\fraku\\ j\neq i}} \abs{\ell_j}^{-\alpha}\right) 
 \sum_{t=1}^\infty \sum_{\ell_i=t 2^{\max (0,m-w_i)}}^{(t+1) 2^{\max (0,m-w_i)}-1}
 \left[\frac{\delta_N (\ell_i Y_i z_i + b)}{\abs{\ell_i}^\alpha} + \frac{\delta_N (\ell_i Y_i z_i - b)}{\abs{\ell_i}^\alpha}\right]\\
 &\le& \gamma_{\fraku} \left(\prod_{\substack{j\in\fraku\\ j\neq i}} \abs{\ell_j}^{-\alpha}\right) 
 \sum_{t=1}^\infty \frac{1}{(t2^{\max (0,m-w_i)})^\alpha}
 \sum_{\ell_i=t 2^{\max (0,m-w_i)}}^{(t+1) 2^{\max (0,m-w_i)}-1}
 \left(\delta_N (\ell_i Y_i z_i + b)+\delta_N (\ell_i Y_i z_i - b)\right)\\
 &=& \gamma_{\fraku} \left(\prod_{\substack{j\in\fraku\\ j\neq i}} \abs{\ell_j}^{-\alpha}\right) 
 \frac{1}{2^{\alpha \max (0,m-w_i)}}
 \sum_{t=1}^\infty \frac{1}{t^\alpha}
 \sum_{\ell_i=t 2^{\max (0,m-w_i)}}^{(t+1) 2^{\max (0,m-w_i)}-1}
 \left(\delta_N (\ell_i Y_i z_i + b)+\delta_N (\ell_i Y_i z_i - b)\right)\\ 
 &=&  \gamma_{\fraku} \left(\prod_{\substack{j\in\fraku\\ j\neq i}} \abs{\ell_j}^{-\alpha}\right) 
 \frac{2\zeta (\alpha)}{2^{\alpha \max (0,m-w_i)}},
\end{eqnarray*} where the last equality follows from the fact that
\[ 
  \sum_{\ell_i=t 2^{\max (0,m-w_i)}}^{(t+1) 2^{\max (0,m-w_i)}-1}\delta_N (\ell_i Y_i z_i + b) = 1
\] 
holds since the congruence $\ell_i Y_i z_i + b = \ell_i 2^{w_i} z_i + b \equiv 0 \tmod{N}$ is equivalent to 
$\ell_i 2^{w_i} \equiv - z_i^{-1} b \tmod{N}$, and then the latter congruence can have at most one solution 
$\ell_i$ in $\{t 2^{\max (0,m-w_i)},\ldots, (t+1) 2^{\max (0,m-w_i)}-1 \}$, see also 
\cite[Corollary of Proposition 1]{Kor63}. Therefore, we can estimate $T_{\setu,\bsw_\setu}$ for $\abs{u}>1$, by
\begin{eqnarray*}
 T_{\fraku,\bsw_{\fraku}} &\le & \gamma_{\fraku} 2\zeta (\alpha) \sum_{i\in\fraku} \frac{1}{2^{\alpha \max (0,m-w_i)}}
 \sum_{\bsell_{\fraku \setminus \{i\}}\in\ZZ_*^{\abs{\fraku}-1}}\prod_{\substack{j\in\fraku\\ j\neq i}} \abs{\ell_j}^{-\alpha}\\
 &=& \gamma_{\fraku} 2\zeta (\alpha) \sum_{i\in\fraku} \frac{1}{2^{\alpha \max (0,m-w_i)}}
 \left(2\sum_{\ell=1}^\infty \frac{1}{\ell^\alpha}\right)^{\abs{\fraku}-1}\\
 &=& (2\zeta (\alpha))^{\abs{\fraku}} \gamma_{\fraku} \sum_{i\in\fraku} \frac{1}{2^{\alpha \max (0,m-w_i)}}\\
 &\le& \frac{(2\zeta (\alpha))^{\abs{\fraku}} \gamma_{\fraku} \abs{\fraku}}{2^{\alpha \max (0,m-\max_{i\in\fraku} w_i)}}\\
 &\le& \gamma_{\fraku}\frac{(4\zeta (\alpha))^{\abs{\fraku}}}{2^{\alpha \max (0,m-\max_{i\in\fraku} w_i)}}.
\end{eqnarray*}

In summary we obtain, by the results for both cases from above,
\[
 \sum_{\emptyset\neq\fraku \subseteq [d]} T_{\fraku,\bsw_{\fraku}} 
 \le \sum_{\emptyset\neq\fraku \subseteq [d]}
\gamma_\setu \frac{(4\zeta(\alpha))^{\abs{\fraku}}}{2^{\alpha\max(0,m-\max_{j\in\fraku} w_j)}},
\]
as claimed. 
\end{proof}

We have obtained a key ingredient for the worst-case error analysis 
from the result on the truncation error in Proposition \ref{prop:trunc_error}.

\section{The construction method for rank-1 lattice rules}\label{sec:construction}
In this section, we will introduce, formulate, and analyze a method for the construction of lattice rules. 
Firstly, we want to be able to estimate efficiently the quantity $T_\bsgamma(N,\bsz,\bsw)$ in \eqref{eq:qual_meas}, which is
needed for our construction method, and to do this we summarize some auxiliary statements which will be needed in the following analysis.

\subsection{Preliminary results}
The following lemma (see \cite{EKNO21} for a proof) shows that the function $\log (\sin^{-2}(\pi x) )$
can be written in terms of its truncated Fourier series with uniformly bounded remainder term. 
We also bear in mind that it cannot be evaluated in $x= 0$ and $x=1$.
\begin{lemma}\label{lemma:expr_target_function}
	Let $N \in \N$, then for any $x \in (0,1)$ there exists a $\tau(x) \in \R$ with $|\tau(x)| \le 1$ such that 
	\begin{equation*}
		\log (\sin^{-2}(\pi x) )
		=
		\log 4 + \sum_{\substack{\ell = -(N-1)\\ \ell\neq 0}}^{N-1} 
		\frac{\rme^{2\pi \icomp \ell x}}{|\ell|} + \frac{\tau(x)}{N \|x\|}
		=
		\sum_{\ell = -(N-1)}^{N-1} \frac{\rme^{2\pi \icomp \ell x}}{b(\ell)} + \frac{\tau(x)}{N \|x\|},
	\end{equation*}
	with coefficients
	\begin{align} \label{eq:def-b-dnint}
		b(\ell)
		&:= 
		\left\{\begin{array}{cc}
		\abs{\ell}, & {\text{for }} \ell \ne 0 , \\ 
		1/\log4, & {\text{for }} \ell = 0 ,
		\end{array}\right.
	\end{align}
   where $\|x\|$ denotes the distance to the nearest integer of $x$, i.e.,
   \[
		\|x\|
		:=
		\min\{\{x\},1-\{x\}\}.
   \]
\end{lemma}
What is more, we will make use of the following lemma, which was also proved in \cite{EKNO21}.
\begin{lemma} \label{lemma:diff_prod}
	For $j \in [d]$, let $u_j, v_j, \overline{u}_j$, and $r_j$ be real numbers which satisfy
	\begin{align*}
		(a) \quad u_j = v_j + r_j, \quad
		(b) \quad |u_j| \le \bar{u}_j, \quad
		(c) \quad \bar{u}_j \geq 1 .
	\end{align*}
	Then, for any subset $\emptyset \ne \setu \subseteq [d]$ there exists a $\theta_{\setu}$ with $|\theta_{\setu}| \le 1$ such that the following identity holds,
	\begin{align}\label{eq:diff_prod} 
		\prod_{j \in \setu} u_j
		&=
		\prod_{j \in \setu} v_j + \theta_{\setu} \left(\prod_{j \in \setu} (\bar{u}_j+|r_j|) \right) \sum_{j \in \setu} |r_j| .
	\end{align}
\end{lemma}
\subsection{The reduced CBC-DBD construction}
We are now ready to study the reduced component-by-component digit-by-digit (CBC-DBD) construction for lattice rules. 
To this end, we assume again throughout the section that $N$ is of the form $2^m$ for some positive integer $m$; 
this choice of $N$ is natural as the components of the generating vector will be constructed digit by digit, i.e., 
bit by bit as we consider $N$ to be a power of~$2$. We first show the following estimate on the quantity $T_\bsgamma(N,\bsz,\bsw)$, 
which already indicates the target function to be minimized in Algorithm \ref{alg:redCBCDBD} below.

\begin{theorem} \label{thm:T_target_CBCDBD}
Let $N=2^m$, with $m>3$, let $\bsgamma=(\gamma_j)_{j\ge 1}$ be positive weights, 
and let $\bsw = (w_j)_{j\ge 1}$ be a sequence of reduction indices in $\N_0$ with $0=w_1\le w_2 \le \cdots$. 
Furthermore, let $Y_j:= 2^{w_j}$ for $j \in [d]$, and let $\bsz = (Y_1z_1, \ldots, Y_dz_d)\in\{1,\ldots,N-1\}^d$, 
where we assume that all $z_j$, $j\in [d]$, are odd, and that $z_j\in\{1,3,\ldots,2^{m-w_j}-1\}$ for $j\in [d^*]$. 
Then the following estimate holds,
\begin{align*}
	T_\bsgamma(N,\bsz,\bsw) 
	&\le\sum_{\emptyset\neq \fraku\subseteq [d^*] } \frac{\gamma_{\fraku}}{N}
 2^{w_{j_{\abs{\fraku}}}+1}  (6\log N)^{\abs{\fraku}} (1+\log N)\\
 &+ \sum_{\emptyset\neq \fraku\subseteq [d^*] } \frac{\gamma_{\fraku}}{N} H_{N,\fraku,\bsw}
 - \sum_{\emptyset\neq \fraku\subseteq [d^*] } \gamma_{\fraku} (\log 4)^{\abs{\fraku}}
	,
\end{align*}
where $j_{\abs{\fraku}}$ denotes the largest element of $\fraku$ for $\emptyset\neq \fraku\subseteq [d^*] $,
where $d^*:= \max\{j\in \N: w_j <m\}$, and where 
\[H_{N,\fraku,\bsw}:= \sum_{t=0}^{2^{w_{j_{\abs{\fraku}}}}-1}\ \ 
\sum_{k=t 2^{m-w_{j_{\abs{\fraku}}}} +1}^{(t+1)2^{m-w_{j_{\abs{\fraku}}}}-1}\prod_{j\in\fraku} \log \left(\frac{1}{\sin^2 (\pi z_j k / 2^{m-w_j})}\right).\]
\end{theorem}

\begin{proof}
We have
\begin{eqnarray*}
 T_{\bsgamma}(N,\bsz,\bsw)&=& \sum_{\bszero\neq\bsell \in M_{N,d,\bsw}} \frac{\delta_N (\bsell\cdot\bsz)}{r_{1,\bsgamma} (\bsell)}\\
 &=&\sum_{\emptyset\neq\fraku \subseteq [d]} \gamma_{\fraku} \sum_{\bsell_{\fraku}\in M_{N,\abs{\fraku},\bsw_{\fraku}}^*}
 \frac{\delta_N (\bsell_{\fraku}\cdot\bsz_{\fraku})}{\prod_{j\in\fraku} \abs{\ell_j}}.
\end{eqnarray*} Note that $M_{N,\abs{\fraku},\bsw_{\fraku}}^*=\emptyset$ if 
$\fraku \cap \{d^* +1, d^* +2,\ldots,d\}\neq \emptyset$. 
Therefore, recalling the definition of $b(\ell)$ in~\eqref{eq:def-b-dnint},
\begin{eqnarray}\label{eqn:Testimate}\notag
T_{\bsgamma}(N,\bsz,\bsw)&=&
 \sum_{\emptyset\neq\fraku \subseteq [d^*]} \gamma_{\fraku} \sum_{\bsell_{\fraku}\in M_{N,\abs{\fraku},\bsw_{\fraku}}^*}
 \frac{\delta_N (\bsell_{\fraku}\cdot\bsz_{\fraku})}{\prod_{j\in\fraku} \abs{\ell_j}}\\\notag
 &=& \sum_{\emptyset\neq\fraku \subseteq [d^*]} \gamma_{\fraku} \sum_{\bsell_{\fraku}\in M_{N,\abs{\fraku},\bsw_{\fraku}}^*}
 \frac{\delta_N (\bsell_{\fraku}\cdot\bsz_{\fraku})}{\prod_{j\in\fraku} b(\ell_j)}\\\notag
 &\le & \sum_{\emptyset\neq\fraku \subseteq [d^*]} \gamma_{\fraku} \sum_{\bszero\neq\bsell_{\fraku}\in M_{N,\abs{\fraku},\bsw_{\fraku}}}
 \frac{\delta_N (\bsell_{\fraku}\cdot\bsz_{\fraku})}{\prod_{j\in\fraku} b(\ell_j)}\\\notag
 &=& \sum_{\emptyset\neq\fraku \subseteq [d^*]} \frac{\gamma_{\fraku}}{N} \sum_{k=0}^{N-1} 
 \left(\sum_{\bsell_{\fraku}\in M_{N,\abs{\fraku},\bsw_{\fraku}}}
 \frac{e^{2\pi\icomp k \bsell_{\fraku}\cdot \bsz_{\fraku}/N}}{\prod_{j\in\fraku} b(\ell_j)}-(\log 4)^{\abs{\fraku}}\right)\\\notag
 &=& \sum_{\emptyset\neq\fraku \subseteq [d^*]} \frac{\gamma_{\fraku}}{N} \sum_{k=0}^{N-1} 
 \sum_{\bsell_{\fraku}\in M_{N,\abs{\fraku},\bsw_{\fraku}}}
 \frac{e^{2\pi\icomp k \bsell_{\fraku}\cdot \bsz_{\fraku}/N}}{\prod_{j\in\fraku} b(\ell_j)}
 -\sum_{\emptyset\neq\fraku \subseteq [d^*]}\gamma_\setu(\log 4)^{\abs{\fraku}}\\\notag
 &=& \sum_{\emptyset\neq\fraku \subseteq [d^*]} \frac{\gamma_{\fraku}}{N} \sum_{k=0}^{N-1}  
 \prod_{j\in\fraku} \left(\log 4 + \sum_{\ell_j \in M_{N,j,w_{j}}^*} \frac{e^{2\pi\icomp k \ell_j 2^{w_j} z_j / 2^m}}{\abs{\ell_j}}\right)
 -\sum_{\emptyset\neq\fraku \subseteq [d^*]}\gamma_\setu(\log 4)^{\abs{\fraku}}\\ \notag
&=& \sum_{\emptyset\neq\fraku \subseteq [d^*]} \frac{\gamma_{\fraku}}{N} \sum_{k=0}^{N-1}  
\prod_{j\in\fraku} \left(\log 4 + \sum_{\ell_j \in M_{N,j,w_{j}}^*} \frac{e^{2\pi\icomp k \ell_j z_j / 2^{m-w_j}}}{\abs{\ell_j}}\right)
 -\sum_{\emptyset\neq\fraku \subseteq [d^*]}\gamma_\setu (\log 4)^{\abs{\fraku}}.\\ 
\end{eqnarray}
We can analyze the first term in \eqref{eqn:Testimate} as follows. We assume that
\[
 \fraku =\{j_1,j_2,\ldots,j_{\abs{\fraku}}\},
\]
such that $1\le j_1 < j_2 < \cdots < j_{\abs{\fraku}}$. This implies
\[
 m\ge m - w_{j_1} \ge m - w_{j_2} \ge \cdots \ge m - w_{j_{\abs{\fraku}}}.
\]
Suppose now that $k\in\{0,1,\ldots,N-1\}$ is such that 
$k\not\equiv 0 \text{ (mod ${2^{m-w_{j_{\abs{\fraku}}}}}$)}.$ This also implies 
\[
 k\not\equiv 0 \text{ (mod $2^{m-w_{j_i}}$)}\quad \forall i\in\{1,\ldots,\abs{\fraku}\}.
\]
Therefore, for fixed $\emptyset\neq\fraku \subseteq [d^*]$, we can rewrite the inner sum over $k$ of the first term in \eqref{eqn:Testimate} as
\begin{eqnarray}\label{eqn:qualitT}\notag
\lefteqn{\sum_{k=0}^{N-1}  \prod_{j\in\fraku} \left(\log 4 + \sum_{\ell_j \in M_{N,j,w_{j}}^*}  
\frac{e^{2\pi\icomp k \ell_j z_j / 2^{m-w_j}}}{\abs{\ell_j}}\right)} \\
\notag &=&  \sum_{t=0}^{2^{w_{j_{\abs{\fraku}}}}-1}
\ \ \sum_{k=t 2^{m-w_{j_{\abs{\fraku}}}} +1}^{(t+1)2^{m-w_{j_{\abs{\fraku}}}}-1}
 \left(\prod_{j\in\fraku} \left(\log 4 + \sum_{\ell_j \in M_{N,j,w_{j}}^*} 
 \frac{e^{2\pi\icomp k \ell_j z_j / 2^{m-w_j}}}{\abs{\ell_j}}\right)\right)\\
\notag  && + \sum_{t=0}^{2^{w_{j_{\abs{\fraku}}}}-1}
 \left(\prod_{j\in\fraku} \left(\log 4 + \sum_{\ell_j \in M_{N,j,w_{j}}^*} 
 \frac{e^{2\pi\icomp t 2^{m-w_{j_{\abs{\fraku}}}}  \ell_j z_j / 2^{m-w_j}}}{\abs{\ell_j}}\right)\right)\\
 &=:& \Sigma_{\fraku}^{(1)} + \Sigma_{\fraku}^{(2)}. 
\end{eqnarray} 
Now we estimate the second summand $ \Sigma_{\fraku}^{(2)}$ in \eqref{eqn:qualitT}  as follows,
\begin{eqnarray*}
 \Sigma_{\fraku}^{(2)} \le  \abs{\Sigma_{\fraku}^{(2)}}&\le & \sum_{t=0}^{2^{w_{j_{\abs{\fraku}}}}-1}
 \prod_{j\in\fraku} \left(\log 4 + \sum_{\ell_j \in M_{N,j,w_{j}}^*} \frac{1}{\abs{\ell_j}}\right)\\
 &=& \sum_{t=0}^{2^{w_{j_{\abs{\fraku}}}}-1}
 \prod_{j\in\fraku} \left(\log 4 + 2\sum_{\ell =1}^{2^{m-w_j}-1} \frac{1}{\ell}\right)\\
 &\le& \sum_{t=0}^{2^{w_{j_{\abs{\fraku}}}}-1}
 \prod_{j\in\fraku} \left(\log 4 + 2\sum_{\ell =1}^{N-1} \frac{1}{\ell}\right)\\
 &\le& \sum_{t=0}^{2^{w_{j_{\abs{\fraku}}}}-1}
 \prod_{j\in\fraku} \left(\log 4 + 4\log N\right)\\
  &\le& \sum_{t=0}^{2^{w_{j_{\abs{\fraku}}}}-1} (6\log N)^{\abs{\fraku}}\\
  &=& 2^{w_{j_{\abs{\fraku}}}} (6\log N)^{\abs{\fraku}},
\end{eqnarray*}
where we implicitly assumed that $N\ge 3$ (which is no significant restriction). 
Let us now analyze $ \Sigma_{\fraku}^{(1)} $ in \eqref{eqn:qualitT}, where we proceed similarly to \cite{EKNO21}. For $j\in \fraku$ 
we get
\begin{eqnarray}\label{used_product_lemma}
 \Sigma_{\fraku}^{(1)} &=&  \notag
 \sum_{t=0}^{2^{w_{j_{\abs{\fraku}}}}-1}\ \ \sum_{k=t 2^{m-w_{j_{\abs{\fraku}}}} +1}^{(t+1)2^{m-w_{j_{\abs{\fraku}}}}-1}
 \left(\prod_{j\in\fraku} v_j (k) - \prod_{j\in\fraku} u_j (k) + \prod_{j\in\fraku} u_j (k)\right)\\\notag
 &=& \sum_{t=0}^{2^{w_{j_{\abs{\fraku}}}}-1}\ \ \sum_{k=t 2^{m-w_{j_{\abs{\fraku}}}} +1}^{(t+1)2^{m-w_{j_{\abs{\fraku}}}}-1}
 \prod_{j\in\fraku} u_j (k)\\\notag
 &&+\sum_{t=0}^{2^{w_{j_{\abs{\fraku}}}}-1}\ \ \sum_{k=t 2^{m-w_{j_{\abs{\fraku}}}} +1}^{(t+1)2^{m-w_{j_{\abs{\fraku}}}}-1}
 \left(\prod_{j\in\fraku} v_j (k) - \prod_{j\in\fraku} u_j (k) \right)\\\notag
 &=& \sum_{t=0}^{2^{w_{j_{\abs{\fraku}}}}-1}\ \ \sum_{k=t 2^{m-w_{j_{\abs{\fraku}}}} +1}^{(t+1)2^{m-w_{j_{\abs{\fraku}}}}-1}
 \prod_{j\in\fraku} u_j (k)\\\label{eq:T-bound-ln-sin}
 &&+\sum_{t=0}^{2^{w_{j_{\abs{\fraku}}}}-1}\ \ \sum_{k=t 2^{m-w_{j_{\abs{\fraku}}}} +1}^{(t+1)2^{m-w_{j_{\abs{\fraku}}}}-1}
(-\theta_{\fraku} (k)) \left(\prod_{j\in\fraku} (\overline{u}_j (k) + \abs{r_j (k)})\right) \sum_{j\in\fraku} \abs{r_j (k)},
\end{eqnarray}
where in \eqref{used_product_lemma} we used Lemma \ref{lemma:diff_prod} with
\begin{align*}
&u_j=u_j (k):=\log \left(\frac{1}{\sin^2 (\pi z_j k / 2^{m-w_j})}\right),\quad\qquad\qquad \overline{u}_j = \overline{u}_j (k):= 2\log N,\\
&v_j=v_j (k):=\log 4 +\sum_{\ell_j \in M_{N,j,w_{j}}^*} \frac{e^{2\pi\icomp k \ell_j z_j / 2^{m-w_j}}}{\abs{\ell_j}}, \quad
r_j=r_j (k):= \frac{\tau_j (k)}{2^{m-w_j} \norm{z_j k / 2^{m-w_j}}},
\end{align*}
where the terms $\theta_{\fraku} (k)$ are defined analogously to Lemma \ref{lemma:diff_prod}, 
and satisfy $\abs{\theta_{\fraku} (k)}\le 1$ and the $\tau_j (k)$ are analogous to Lemma \ref{lemma:expr_target_function}, and
also satisfy $\abs{\tau_j (k)}\le 1$. It can be checked easily that the conditions in the lemmas are fulfilled. 
Indeed, Condition (a) in Lemma \ref{lemma:diff_prod} is satisfied due to Lemma \ref{lemma:expr_target_function}. 
Furthermore, as in \cite[Proof of Theorem 2]{EKNO21}, we see that 
\[
 \sin^2 \left(\pi \frac{z_j k}{2^{m-w_j}}\right) \ge \left(\frac{1}{2^{m-w_j}}\right)^2.
\]
Therefore,
\[
 \abs{u_j(k)}=\log \left(\frac{1}{\sin^2 (\pi z_j k / 2^{m-w_j})}\right) \le \log \left(\left(2^{m-w_j}\right)^2\right)
 \le 2\log N =\overline{u}_j (k). 
\]
Furthermore, $\overline{u}_j (k)\ge 1$, as long as $N\ge 2$, which again is not a real restriction.
Next we show how to bound the second summand in \eqref{eq:T-bound-ln-sin} independently of the choice of $\bsz$ as follows,
\begin{eqnarray*}
& & 
 \sum_{t=0}^{2^{w_{j_{\abs{\fraku}}}}-1}\ \ \sum_{k=t 2^{m-w_{j_{\abs{\fraku}}}} +1}^{(t+1)2^{m-w_{j_{\abs{\fraku}}}}-1}
(-\theta_{\fraku} (k)) \left(\prod_{j\in\fraku} (\overline{u}_j (k) + \abs{r_j (k)})\right) \sum_{j\in\fraku} \abs{r_j (k)}\\
&\le & \sum_{t=0}^{2^{w_{j_{\abs{\fraku}}}}-1}\ \ \sum_{k=t 2^{m-w_{j_{\abs{\fraku}}}} +1}^{(t+1)2^{m-w_{j_{\abs{\fraku}}}}-1}
\abs{\theta_{\fraku} (k)}
\left(\prod_{j\in\fraku} \left(2\log N + \frac{\abs{\tau_j (k)}}{2^{m-w_j} \norm{z_j k / 2^{m-w_j}}} \right)\right)\\
&&\times \sum_{j\in\fraku} \frac{\abs{\tau_j (k)}}{2^{m-w_j} \norm{z_j k / 2^{m-w_j}}}\\
&\le & \sum_{t=0}^{2^{w_{j_{\abs{\fraku}}}}-1}\ \ \sum_{k=t 2^{m-w_{j_{\abs{\fraku}}}} +1}^{(t+1)2^{m-w_{j_{\abs{\fraku}}}}-1}
\abs{\theta_{\fraku} (k)}
\left(\prod_{j\in\fraku} (1+ 2\log N )\right)
\sum_{j\in\fraku} \frac{\abs{\tau_j (k)}}{2^{m-w_j} \norm{z_j k / 2^{m-w_j}}}\\
&=& \left(\prod_{j\in\fraku} (1+ 2\log N )\right)\sum_{j\in\fraku} \sum_{t=0}^{2^{w_{j_{\abs{\fraku}}}}-1}
\ \ \sum_{k=t 2^{m-w_{j_{\abs{\fraku}}}} +1}^{(t+1)2^{m-w_{j_{\abs{\fraku}}}}-1}
\frac{\abs{\theta_{\fraku} (k)}\abs{\tau_j (k)}}{2^{m-w_j} \norm{z_j k / 2^{m-w_j}}}\\
&\leq& (1+2\log N)^{\abs{\fraku}} \sum_{j\in\fraku}
\sum_{t=0}^{2^{w_{j_{\abs{\fraku}}}}-1}\ \ \sum_{k=t 2^{m-w_{j_{\abs{\fraku}}}} +1}^{(t+1)2^{m-w_{j_{\abs{\fraku}}}}-1}
\frac{1}{2^{m-w_j} \norm{z_j k / 2^{m-w_j}}}\\
&\le& (1+2\log N)^{\abs{\fraku}} \sum_{j\in\fraku}
\sum_{t=0}^{2^{w_{j_{\abs{\fraku}}}}-1}\ \ \sum_{k=1}^{2^{m-w_j}-1}
\frac{1}{2^{m-w_j} \norm{z_j k / 2^{m-w_j}}}\\ &\le& (1+2\log N)^{\abs{\fraku}} 2\abs{\fraku} 
 2^{w_{j_{\abs{\fraku}}}} (1+\log N) \\
 &\le& (1+2\log N)^{\abs{\fraku}} 2^{\abs{\fraku}} 
 2^{w_{j_{\abs{\fraku}}}} (1+\log N) \\
 &\le& (6\log N)^{\abs{\fraku}} 
 2^{w_{j_{\abs{\fraku}}}} (1+\log N),
\end{eqnarray*}
where we used that
\[2^{m-w_j} \norm{\frac{z_j k}{2^{m-w_j}}}\ge2^{m-w_j}\norm{\frac{1}{2^{m-w_j}}} =1,
 \]
that $2\abs{\fraku} \le 2^{\abs{\fraku}}$, and that $2^{m-w_j}\ge 2^{m-w_{j_{\abs{\fraku}}}}$ for $j\in \fraku$, and hence, if $k$ runs through 
the integers $\{t 2^{m-w_{j_{\abs{\fraku}}}} +1,\ldots, (t+1)2^{m-w_{j_{\abs{\fraku}}}}-1\}$, the values of $ \norm{z_j k / 2^{m-w_j}}$ 
are a subset of the values of $ \norm{z_j k / 2^{m-w_j}}$ when $k$ runs through the integers $\{1,\ldots,2^{m-w_j}-1\}$. Then, due to 
\cite[Corollary of Proposition 4]{Kor63}, we know that 
\[
 \sum_{k=1}^{2^{m-w_j}-1} \frac{1}{2^{m-w_j} \norm{z_j k / 2^{m-w_j}}}
 \le 2 (1+\log 2^{m-w_j}) \le 2 (1+\log N).
\]
Combining the estimates for $ \Sigma_{\fraku}^{(1)} $ and $\Sigma_{\fraku}^{(2)}$, this yields 
\begin{eqnarray*}
\Sigma_{\fraku}^{(1)} + \Sigma_{\fraku}^{(2)}&=& 2^{w_{j_{\abs{\fraku}}}}  (6\log N)^{\abs{\fraku}}  + 
 H_{N,\fraku,\bsw} + 
 (6\log N)^{\abs{\fraku}}  2^{w_{j_{\abs{\fraku}}}} (1+\log N)\\
 &\le & 2^{w_{j_{\abs{\fraku}}}+1}  (6\log N)^{\abs{\fraku}} (1+\log N) +  
 H_{N,\fraku,\bsw}, 
\end{eqnarray*}
where 
\[
H_{N,\setu,\bsw} = \sum_{t=0}^{2^{w_{j_{\abs{\fraku}}}}-1}\ \ \sum_{k=t 2^{m-w_{j_{\abs{\fraku}}}} +1}^{(t+1)2^{m-w_{j_{\abs{\fraku}}}}-1}
 \prod_{j\in\fraku} u_j (k).
 \]
Hence, 
\begin{eqnarray*}
 T_{\bsgamma}(N,\bsg,\bsw) &\le & \sum_{\emptyset\neq \fraku\subseteq [d^*] } \frac{\gamma_{\fraku}}{N}
 2^{w_{j_{\abs{\fraku}}}+1}  (6\log N)^{\abs{\fraku}} (1+\log N)\\
 &&+ \sum_{\emptyset\neq \fraku\subseteq [d^*] } \frac{\gamma_{\fraku}}{N} H_{N,\fraku,\bsw}
 - \sum_{\emptyset\neq \fraku\subseteq [d^*] } \gamma_{\fraku} (\log 4)^{\abs{\fraku}}.
\end{eqnarray*}
\end{proof}

In the following, we write 
\[
  H_{s,N,\bsgamma,\bsw}:= 
  \sum_{\emptyset\neq \fraku\subseteq [s] } \gamma_{\fraku} H_{N,\fraku,\bsw}.
\]

The next lemma motivates the choice of our quality function for the algorithm. In particular, assuming we have already fixed 
the first (i.e., the least significant) $v-1$ bits of $z_s$ for $s\in [d^*]$, we would like to find out how good a specific choice for the $v$th bit is, 
in terms of $H_{s,N,\bsgamma,\bsw}$. To this end, we consider the average over all remaining $m-w_s-v$ bits.

For the base 2-digit representation of a $z_s$ with $1\le s\le d$, we write 
\[
 z_s=z_s^{(0)} + z_s^{(1)}2 + z_s^{(2)} 2^2 + \cdots .
\]

\begin{lemma}\label{avg_sum}
For an integer $v\in \{1,\dots,m-w_s-1\}$, with $m \in \N$ and $s\in [d^*]$, let $z\in\{0,1\}$ and $\bsz =(z_1,\ldots,z_s) \in \Z^s$, 
and where the first $v-1$ bits of $z_s$ have been selected. 
We write $z_{s,v-1}:=z_{s}^{(0)} + z_{s}^{(1)} 2 + \cdots + z_{s}^{(v-2)} 2^{v-2}$, 
and $\widetilde{z}_s:=z_{s}^{(0)} + z_{s}^{(1)}  2 + \cdots + z_{s}^{(v-2)} 2^{v-2} + z 2^{v-1}$. 
Then the average of $H_{s,N,\bsgamma,\bsw}$ over the next $m-w_s-v$ 
bit choices for $z_s$ is given by
\begin{align}\notag
    &\frac{1}{2^{m-w_s-v}} \sum_{\overline{z}\in \ZZ_{2^{m-w_s-v}}} H_{s,N,\bsgamma,\bsw} (z_1,\ldots,z_{s-1}, \widetilde{z}_s
    + \overline{z}2^v)\\\notag &=\sum_{t=v}^{m-w_s} \frac{1}{2^{t-v}} 
    \sum_{\ell=0}^{2^{w_s}-1} \ \ \sum_{\substack{k=\ell 2^t + 1\\ k\equiv 1 \tmod{2}}}^{( \ell +1) 2^t-1}
    \sum_{\substack{\emptyset\neq \fraku\subseteq [s]\\ s\in\fraku }} \gamma_{\fraku}\log 
    \left(\frac{1}{\sin^2 (\pi k \widetilde{z}_s  / 2^v)}\right)\\ \notag
    &\times\prod_{\substack{j\in\fraku\\ j\neq s}} 
     \log \left(\frac{1}{\sin^2 \left(\pi z_j \frac{k}{2^t}\,\frac{2^{m-w_s}}{2^{m-w_j}}\right)}\right)\\
     \label{eqn:hfunc}&+ S_{N,v,\bsgamma,\bsw}(\bsz),
 \end{align} 
 where the term $S_{N,v,\bsgamma,\bsw}(\bsz)$, which is independent of $z$ and $\bar{z}$, is given by
\begin{align*}
S_{N,v,\bsgamma,\bsw}(\bsz) &= \sum_{r=1}^{s-1} \sum_{\substack{k=0\\ k\not\equiv 0 \tmod{2^{m-w_{r}}}}}^{N-1}
 \sum_{\substack{\emptyset\neq \fraku\subseteq [r]\\ r\in\fraku }} \gamma_{\fraku}
 \prod_{j\in\fraku} \log \left(\frac{1}{\sin^2 (\pi z_j k / 2^{m-w_j})}\right)\\
 &+\sum_{ \ell=0}^{2^{w_s}-1} \ \ \sum_{t=1}^{v-1} \sum_{\substack{k= \ell 2^t + 1\\ k\equiv 1 \tmod{2}}}^{( \ell +1) 2^t-1}
 \sum_{\substack{\emptyset\neq \fraku\subseteq [s]\\ s\in\fraku }} \gamma_{\fraku}
  \left(\prod_{\substack{j\in\fraku\\ j\neq s}} 
 \log \left(\frac{1}{\sin^2 \left(\pi z_j \frac{k}{2^t}\,\frac{2^{m-w_s}}{2^{m-w_j}}\right)}\right)\right)\\
 &\times 
 \log\left(\frac{1}{\sin^2 (\pi z_{s,v-1} k / 2^t)}\right)\\ 
 &+ \sum_{\ell=0}^{2^{w_s}-1} \ \ \sum_{t=v}^{m-w_s}\frac{2^{t-v}-1}{2^{t-v}} (\log 4 ) 
 \sum_{\substack{k= \ell 2^t + 1\\ k\equiv 1 \tmod{2}}}^{( \ell +1) 2^t-1}
 \sum_{\substack{\emptyset\neq \fraku\subseteq [s]\\ s\in\fraku }} 
 \gamma_{\fraku}\prod_{\substack{j\in\fraku\\ j\neq s}} 
 \log \left(\frac{1}{\sin^2 \left(\pi z_j \frac{k}{2^t}\,\frac{2^{m-w_s}}{2^{m-w_j}}\right)}\right).
\end{align*}
\end{lemma}
\begin{proof}
For arbitrary $\bsz=(z_1,\ldots,z_s)$, we can rewrite the quantity $H_{s,N,\bsgamma,\bsw}$ as follows,
\begin{eqnarray}\notag
 H_{s,N,\bsgamma,\bsw}&=&H_{s,N,\bsgamma,\bsw}(\bsz)\\\notag
 &:=& N \sum_{\emptyset\neq \fraku\subseteq [s] } \frac{\gamma_{\fraku}}{N} H_{N,\fraku,\bsw}\\\notag
 &=&\sum_{\emptyset\neq \fraku\subseteq [s] } \gamma_{\fraku} 
 \sum_{t=0}^{2^{w_{j_{\abs{\fraku}}}}-1}\ \ \sum_{k=t 2^{m-w_{j_{\abs{\fraku}}}} +1}^{(t+1)2^{m-w_{j_{\abs{\fraku}}}}-1}
 \prod_{j\in\fraku} u_j (k)\\\notag
 &=&\sum_{\emptyset\neq \fraku\subseteq [s] } \gamma_{\fraku} 
 \sum_{t=0}^{2^{w_{j_{\abs{\fraku}}}}-1}\ \ \sum_{k=t 2^{m-w_{j_{\abs{\fraku}}}} +1}^{(t+1)2^{m-w_{j_{\abs{\fraku}}}}-1}
 \prod_{j\in\fraku} \log \left(\frac{1}{\sin^2 (\pi z_j k / 2^{m-w_j})}\right)\\ \notag
 &=&\sum_{\emptyset\neq \fraku\subseteq [s] } \gamma_{\fraku} \sum_{\substack{k=0\\ k\not\equiv 0 \tmod{2^{m-w_{j_{\abs{\fraku}}}}} }}^{N-1}
 \prod_{j\in\fraku} \log \left(\frac{1}{\sin^2 (\pi z_j k / 2^{m-w_j})}\right)\\\label{H_qty}
 &=&\sum_{r=1}^{s} \sum_{\substack{k=0\\ k\not\equiv 0 \text{ (mod $2^{m-w_{r}}$)}}}^{N-1}
 \sum_{\substack{\emptyset\neq \fraku\subseteq [r]\\ r\in\fraku }} \gamma_{\fraku}
 \prod_{j\in\fraku} \log \left(\frac{1}{\sin^2 (\pi z_j k / 2^{m-w_j})}\right).
\end{eqnarray}
Hence, we would like to estimate
\begin{align}\label{eq_avg_sum}
    \frac{1}{2^{m-w_s-v}} \sum_{\overline{z}\in \ZZ_{2^{m-w_s-v}}} 
\sum_{r=1}^s \sum_{\substack{k=0\\ k\not\equiv 0 \tmod{2^{m-w_{r}}} }}^{N-1}
 \sum_{\substack{\emptyset\neq \fraku\subseteq [r]\\ r\in\fraku }} \gamma_{\fraku}
 \prod_{j\in\fraku} \log \left(\frac{1}{\sin^2 (\pi z_j k / 2^{m-w_j})}\right),
\end{align}
where we now assume that $z_s=z_{s,v-1} + z 2^{v-1} + \overline{z} 2^v=\widetilde{z}_s + \overline{z} 2^v$.
Observe that all terms in the sum over $r\in\{1,2,\ldots,s\}$ for which $1\le r \le s-1$ are independent of 
$\overline{z}$, and from this we obtain the first sum in the definition of $S_{m,v,\bsgamma,\bsw}$.
Therefore we need to analyze the remaining part of \eqref{eq_avg_sum}, which equals
\begin{align*}
     \frac{1}{2^{m-w_s-v}} \sum_{\overline{z}\in \ZZ_{2^{m-w_s-v}}} 
\sum_{\substack{k=0\\ k\not\equiv 0 \tmod{2^{m-w_s}}}}^{N-1}
 \sum_{\substack{\emptyset\neq \fraku\subseteq [s]\\ s\in\fraku }} \gamma_{\fraku}
 \prod_{j\in\fraku} \log \left(\frac{1}{\sin^2 (\pi z_j k / 2^{m-w_j})}\right).
\end{align*}
Now, we are going to use the general identity
\begin{equation}\label{eq:nice_identity}
 \sum_{k=1}^{2^p-1} f(k/2^p)=\sum_{t=1}^p \sum_{\substack{k=1\\ k\equiv 1 \tmod{2}}}^{2^{t}-1} f(k/2^{t}).
\end{equation}
Then, we can write
\begin{eqnarray*}
 \lefteqn{\frac{1}{2^{m-w_s-v}} \sum_{\overline{z}\in \ZZ_{2^{m-w_s-v}}} 
\sum_{\substack{k=0\\ k\not\equiv 0 \tmod{2^{m-w_s}}}}^{N-1}
 \sum_{\substack{\emptyset\neq \fraku\subseteq [s]\\ s\in\fraku }} \gamma_{\fraku}
 \prod_{j\in\fraku} \log \left(\frac{1}{\sin^2 (\pi z_j k / 2^{m-w_j})}\right)}\\
&=& 
\frac{1}{2^{m-w_s-v}} \sum_{\overline{z}\in \ZZ_{2^{m-w_s-v}}} 
\sum_{\ell=0}^{2^{w_s}-1} \ \ \sum_{k=\ell\, 2^{m-w_s}+1}^{(\ell+1)2^{m-w_s}-1}
 \sum_{\substack{\emptyset\neq \fraku\subseteq [s]\\ s\in\fraku }} \gamma_{\fraku}
 \prod_{j\in\fraku} \log \left(\frac{1}{\sin^2 (\pi z_j k / 2^{m-w_j})}\right)\\
&=& 
\frac{1}{2^{m-w_s-v}} \sum_{\overline{z}\in \ZZ_{2^{m-w_s-v}}} 
\sum_{\ell=0}^{2^{w_s}-1} \ \ \sum_{k=1}^{2^{m-w_s}-1}
 \sum_{\substack{\emptyset\neq \fraku\subseteq [s]\\ s\in\fraku }} \gamma_{\fraku}
 \prod_{j\in\fraku} \log \left(\frac{1}{\sin^2 (\pi z_j (k+\ell\, 2^{m-w_s}) / 2^{m-w_j})}\right)\\
&=& 
\frac{1}{2^{m-w_s-v}} \sum_{\overline{z}\in \ZZ_{2^{m-w_s-v}}} 
\sum_{\ell=0}^{2^{w_s}-1} \ \ \sum_{k=1}^{2^{m-w_s}-1}
 \sum_{\substack{\emptyset\neq \fraku\subseteq [s]\\ s\in\fraku }} \gamma_{\fraku}\\
&&\times \prod_{j\in\fraku} \log \left(\frac{1}{\sin^2 (\pi z_j \ell\, 2^{w_j-w_s} + \pi z_j k / 2^{m-w_j})}\right)\\
&=& 
\frac{1}{2^{m-w_s-v}} \sum_{\overline{z}\in \ZZ_{2^{m-w_s-v}}} 
\sum_{\ell=0}^{2^{w_s}-1} \ \ \sum_{k=1}^{2^{m-w_s}-1}
\sum_{\substack{\emptyset\neq \fraku\subseteq [s]\\ s\in\fraku }} \gamma_{\fraku}\\
&&\times \prod_{j\in\fraku} \log \left(\frac{1}{\sin^2 (\pi z_j \ell\, 2^{w_j-w_s} + (\pi z_j k / 2^{m-w_s})2^{w_j-w_s})}\right)\\
&=& 
\frac{1}{2^{m-w_s-v}} \sum_{\overline{z}\in \ZZ_{2^{m-w_s-v}}} 
\sum_{\ell=0}^{2^{w_s}-1} \ \ \sum_{t=1}^{m-w_s} \sum_{\substack{k=1\\ k\equiv 1 \tmod{2}}}^{2^t-1}
\sum_{\substack{\emptyset\neq \fraku\subseteq [s]\\ s\in\fraku }} \gamma_{\fraku}\\
&&\times \prod_{j\in\fraku} \log \left(\frac{1}{\sin^2 (\pi z_j \ell\, 2^{w_j-w_s} + (\pi z_j k / 2^t) 2^{w_j-w_s})}\right)\\
&=& 
\frac{1}{2^{m-w_s-v}} \sum_{\overline{z}\in \ZZ_{2^{m-w_s-v}}} 
\sum_{\ell=0}^{2^{w_s}-1} \ \ \sum_{t=1}^{m-w_s} \sum_{\substack{k=1\\ k\equiv 1 \tmod{2}}}^{2^t-1}
\sum_{\substack{\emptyset\neq \fraku\subseteq [s]\\ s\in\fraku }} \gamma_{\fraku}\\
&&\times \prod_{j\in\fraku} \log\left(\frac{1}{\sin^2 \left(\pi z_j \frac{\ell 2^t +k}{2^t}\,\frac{2^{m-w_s}}{2^{m-w_j}}\right)}\right)\\
&=& 
\frac{1}{2^{m-w_s-v}} \sum_{\overline{z}\in \ZZ_{2^{m-w_s-v}}} 
\sum_{\ell=0}^{2^{w_s}-1} \ \ \sum_{t=1}^{m-w_s} \sum_{\substack{k=\ell\, 2^t + 1\\ k\equiv 1 \tmod{2}}}^{(\ell+1) 2^t-1}
\sum_{\substack{\emptyset\neq \fraku\subseteq [s]\\ s\in\fraku }} \gamma_{\fraku}
\prod_{j\in\fraku} \log \left(\frac{1}{\sin^2 \left(\pi z_j \frac{k}{2^t}\,\frac{2^{m-w_s}}{2^{m-w_j}}\right)}\right)\\
&=&
\sum_{\ell=0}^{2^{w_s}-1} \ \ \sum_{t=1}^{m-w_s} \sum_{\substack{k=\ell 2^t + 1\\ k\equiv 1 \tmod{2}}}^{(\ell+1) 2^t-1}
\sum_{\substack{\emptyset\neq \fraku\subseteq [s]\\ s\in\fraku }} \gamma_{\fraku}
 \left(\prod_{\substack{j\in\fraku\\ j\neq s}} 
\log \left(\frac{1}{\sin^2 \left(\pi z_j \frac{k}{2^t}\,\frac{2^{m-w_s}}{2^{m-w_j}}\right)}\right)\right)\\
&&\times  \frac{1}{2^{m-w_s-v}} \sum_{\overline{z}\in \ZZ_{2^{m-w_s-v}}} 
 \log\left(\frac{1}{\sin^2 (\pi (\widetilde{z}_s + \overline{z} 2^v) k / 2^t)}\right),
\end{eqnarray*} where we remind the reader that 
$\widetilde{z}_s$ = $z_s^{(0)} + z_s^{(1)}2 + \cdots + z_s^{(v-2)}2^{v-2} + z 2^{v-1}$.
Therefore, in the above summation, if $t\in\{1,2,\ldots,v-1\}$, then 
\[
 \sin^2 (\pi (\widetilde{z}_s + \overline{z} 2^v) k / 2^t) 
 =\sin^2 (\pi (z_s^{(0)} + \cdots + z_s^{(v-2)} 2^{v-2}) k / 2^t).
\]
Consequently, we obtain
\begin{eqnarray*}
 \lefteqn{\frac{1}{2^{m-w_s-v}} \sum_{\overline{z}\in \ZZ_{2^{m-w_s-v}}} 
\sum_{\substack{k=0\\ k\not\equiv 0 \tmod{2^{m-w_s}}}}^{N-1}
 \sum_{\substack{\emptyset\neq \fraku\subseteq [s]\\ s\in\fraku }} \gamma_{\fraku}
 \prod_{j\in\fraku} \log \left(\frac{1}{\sin^2 (\pi z_j k / 2^{m-w_j})}\right)}\\ 
&=& 
 \sum_{\ell=0}^{2^{w_s}-1} \ \ \sum_{t=1}^{m-w_s} \sum_{\substack{k=\ell 2^t + 1\\ k\equiv 1 \tmod{2}}}^{(\ell+1) 2^t-1}
 \sum_{\substack{\emptyset\neq \fraku\subseteq [s]\\ s\in\fraku }} \gamma_{\fraku}
 \left(\prod_{\substack{j\in\fraku\\ j\neq s}} 
 \log \left(\frac{1}{\sin^2 \left(\pi z_j \frac{k}{2^t}\,\frac{2^{m-w_s}}{2^{m-w_j}}\right)}\right)\right)\\
 &&\times  \frac{1}{2^{m-w_s-v}} \sum_{\overline{z}\in \ZZ_{2^{m-w_s-v}}} 
 \log\left(\frac{1}{\sin^2 (\pi (\widetilde{z}_s + \overline{z} 2^v) k / 2^t)}\right)\\
&=&
 \sum_{\ell=0}^{2^{w_s}-1} \ \ 
 \sum_{t=1}^{v-1} \sum_{\substack{k=\ell 2^t + 1\\ k\equiv 1 \tmod{2}}}^{(\ell +1) 2^t-1}
 \sum_{\substack{\emptyset\neq \fraku\subseteq [s]\\ s\in\fraku }} \gamma_{\fraku}
 \left(\prod_{\substack{j\in\fraku\\ j\neq s}} 
 \log \left(\frac{1}{\sin^2 \left(\pi z_j \frac{k}{2^t}\,\frac{2^{m-w_s}}{2^{m-w_j}}\right)}\right)\right)\\
 &&\times 
 \log\left(\frac{1}{\sin^2 (\pi (z_s^{(0)} + z_s^{(1)} 2+\cdots + z_s^{(v-2)} 2^{v-2}) k / 2^t)}\right)\\
 &+&\sum_{\ell=0}^{2^{w_s}-1} \ \ \sum_{t=v}^{m-w_s} \sum_{\substack{k=\ell 2^t + 1\\ k\equiv 1 \tmod{2}}}^{(\ell +1) 2^t-1}
 \sum_{\substack{\emptyset\neq \fraku\subseteq [s]\\ s\in\fraku }} \gamma_{\fraku}
 \left(\prod_{\substack{j\in\fraku\\ j\neq s}} 
 \log \left(\frac{1}{\sin^2 \left(\pi z_j \frac{k}{2^t}\,\frac{2^{m-w_s}}{2^{m-w_j}}\right)}\right)\right)\\
 &&\times  \frac{1}{2^{m-w_s-v}} \sum_{\overline{z}\in \ZZ_{2^{m-w_s-v}}} 
 \log\left(\frac{1}{\sin^2 (\pi (\widetilde{z}_s + \overline{z} 2^v) k / 2^t)}\right).
\end{eqnarray*}

For the last expression, we argue as in the proof of Lemma 6 in \cite{EKNO21} to obtain, for $t\in \{v,\dots,m-w_s\}$, 
\begin{eqnarray*}
& &\sum_{\ell=0}^{2^{w_s}-1} \ \ \sum_{t=v}^{m-w_s} \sum_{\substack{k=\ell 2^t + 1\\ k\equiv 1 \tmod{2}}}^{(\ell +1) 2^t-1}
 \sum_{\substack{\emptyset\neq \fraku\subseteq [s]\\ s\in\fraku }} \gamma_{\fraku}
 \prod_{\substack{j\in\fraku\\ j\neq s}} 
 \log \left(\frac{1}{\sin^2 \left(\pi z_j \frac{k}{2^t}\,\frac{2^{m-w_s}}{2^{m-w_j}}\right)}\right)\\
 &&\times  \frac{1}{2^{m-w_s-v}} \sum_{\overline{z}\in \ZZ_{2^{m-w_s-v}}} 
 \log\left(\frac{1}{\sin^2 (\pi (\widetilde{z}_s + \overline{z} 2^v) k / 2^t)}\right)\\&=&
  \sum_{\ell=0}^{2^{w_s}-1} \ \ \sum_{t=v}^{m-w_s} \sum_{\substack{k=\ell 2^t + 1\\ k\equiv 1 \tmod{2}}}^{(\ell +1) 2^t-1}
 \sum_{\substack{\emptyset\neq \fraku\subseteq [s]\\ s\in\fraku }} \gamma_{\fraku}
 \prod_{\substack{j\in\fraku\\ j\neq s}} 
 \log \left(\frac{1}{\sin^2 \left(\pi z_j \frac{k}{2^t}\,\frac{2^{m-w_s}}{2^{m-w_j}}\right)}\right)\\
 &&\times  \frac{2^{m-w_s-t}}{2^{m-w_s-v}} \sum_{\overline{z}\in \ZZ_{2^{t-v}}} 
 \log\left(\frac{1}{\sin^2 (\pi (\widetilde{z}_s + \overline{z} 2^v) k / 2^t)}\right)\\
 &=& 
 \sum_{\ell=0}^{2^{w_s}-1} \ \ \sum_{t=v}^{m-w_s} \sum_{\substack{k=\ell 2^t + 1\\ k\equiv 1 \tmod{2}}}^{(\ell +1) 2^t-1}
 \sum_{\substack{\emptyset\neq \fraku\subseteq [s]\\ s\in\fraku }} \gamma_{\fraku}
 \prod_{\substack{j\in\fraku\\ j\neq s}} 
 \log \left(\frac{1}{\sin^2 \left(\pi z_j \frac{k}{2^t}\,\frac{2^{m-w_s}}{2^{m-w_j}}\right)}\right)\\
 &&\times\left((1-2^{v-t})\log 4 + 2^{v-t}\log \left(\frac{1}{\sin^2 (\pi \widetilde{z}_s k / 2^v)}\right)\right),
\end{eqnarray*}
where in the last step we proceeded exactly as in the proof of Lemma 6 in \cite{EKNO21}. 
This, together with the previous identity yields the claim.
\end{proof}
We observe that in Lemma \ref{avg_sum} only the first term depends on the $v$th bit $z$ of $z_s$, 
while $S_{N,v,\bsgamma,\bsw}(\bsz)$ is independent of this bit. 
This now leads to the introduction of the following digit-wise quality function for the reduced CBC-DBD 
algorithm which is based on the first term in \eqref{eqn:hfunc}. Note that the quality 
function is not exactly equal to the first term in \eqref{eqn:hfunc}, but we add further terms that, 
though independent of the argument of the function, facilitate fast implementation (see Section \ref{sec:fast_impl}).

\begin{definition}\label{def:h_rv}
Let $x\in\N$ be an odd integer, let $m,d^*\in \N$, let $\bsw=(w_j)_{j\ge 1}$ be a sequence of reduction indices in $\N_0$
with $0=w_1\le w_2\le \cdots$, and let $\bsgamma = (\gamma_j)_{j\ge 1}$ 
be a sequence of positive weights. For $1\le s\le d^*$ and $1\le v \le m-w_s$ we define the quality function 
$h_{s,v,m,\bsgamma,\bsw}: \Z \rightarrow \R$ as
\begin{eqnarray*}
 h_{s,v,m,\bsgamma,\bsw} (x):&=& 
 \sum_{t=v}^{m-w_s} \frac{1}{2^{t-v}} 
 \sum_{\ell=0}^{2^{w_s}-1} \ \ \sum_{\substack{k=\ell 2^t + 1\\ k\equiv 1 \tmod{2}}}^{(\ell +1) 2^t-1}
 \sum_{\substack{\fraku\subseteq [s]\\ s\in\fraku }} \gamma_{\fraku\setminus \{s\}} 
 \prod_{\substack{j\in\fraku\\ j\neq s}} 
 \log \left(\frac{1}{\sin^2 \left(\pi z_j \frac{k}{2^t}\,\frac{2^{m-w_s}}{2^{m-w_j}}\right)}\right)\\
 &&\times \left(1+\gamma_s\log \left(\frac{1}{\sin^2 (\pi k x / 2^v)}\right)\right),
\end{eqnarray*}
where we assume that $z_j\in \{1,3,\ldots,2^{m-w_j}-1\}$ is odd for $j\in [s-1]$.
\end{definition}
Based on the quality function $h_{s,v,m,\bsgamma,\bsw}$, we formulate the following reduced component-by-component digit-by-digit 
(CBC-DBD) algorithm.

\begin{algorithm}[H] 
	\caption{Reduced component-by-component digit-by-digit construction}	
	\label{alg:redCBCDBD}
	\vspace{5pt}
	\textbf{Input:} Integer $m \in \N$, dimension $d, \bsw=(w_j)_{j\ge 1}$ with $0=w_1 \le w_2 \le \cdots $ 
	and $Y_j = 2^{w_j}$ for $j \in \{1, \dots , d\}$,
        and positive weights $\bsgamma=(\gamma_j)_{j\ge 1}$. \\
	\vspace{-7pt}
		\begin{algorithmic}
			\STATE Set $z_1=z_{1,m} = 1$ and $z_{1,1} = z_{1,2} = \cdots = z_{d,1} = 1$.
			\STATE If $d> d^*$, set $z_{d^*+1}= \cdots = z_d =1$.
			\vspace{5pt}
			\FOR{$s=2$ \TO $\min \{d,d^*\}$}
			\FOR{$v=2$ \TO $m-w_s$}
			\STATE $z^{\ast} = \underset{z \in \{0,1\}}{\argmin} \; h_{s,v,m,\bsgamma,\bsw}(z_{s,v-1} + z \, 2^{v-1})$
			\STATE $z_{s,v} = z_{s,v-1} + z^{\ast} \, 2^{v-1}$
			\ENDFOR
			\STATE $z_s = z_{s,m-w_s}$
			\ENDFOR
			\vspace{5pt}
			\STATE Set $\bsz = (Y_1z_1,\ldots,Y_dz_d)$.
		\end{algorithmic}
	\vspace{5pt}
	\textbf{Return:} Generating vector $\bsz$ for $N=2^m$.
\end{algorithm}

\subsection{Error convergence behavior of the constructed lattice rules}
In the following, we study the worst-case error behavior of the constructed lattice rules, i.e, we want to show that under 
certain suitable conditions on the weights $\bsgamma$, Algorithm \ref{alg:redCBCDBD} can construct 
generating vectors which yield lattice point sets with very good properties if they are used as integration nodes 
in a QMC rule. For $\bsz=(Y_1z_1,\ldots,Y_dz_d)$, we write $\bsz_{[s]}$ to denote the vector $(z_1,\ldots,z_s)$ for $s\in [d]$.

\begin{theorem}\label{theorem:H-induction}
Let $m\in \N$, $N = 2^m$, $\bsw = (w_j)_{j\ge 1}$ in $\N_0$ with $0=w_1\le w_2\le \cdots $,
and let  $\bsgamma = (\gamma_j)_{j\ge 1}$ be positive product weights. 
Furthermore, let the generating vector $\bsz=(Y_1z_1,\ldots,Y_dz_d) \in \Z^d$ be constructed by Algorithm \ref{alg:redCBCDBD}. Denote 
by $\bsz_{[s]}$ the vector $(z_1,\ldots,z_s)$ for $s\in [d]$.
Then the following estimate holds for $s\in [d^*]$,
\[
H_{s,N,\bsgamma,\bsw}(\bsz_{[s]}) \le (1+\gamma_{s} \log 4) H_{s-1,N,\bsgamma,\bsw}(\bsz_{[s-1]}) 
 + \gamma_{s} (\log 4) (2^m - 2^{w_s}). 
\]
\end{theorem}
\begin{proof}
We will prove the stated estimate via an inductive argument over the selection of the $m-w_s$ bits of the component $z_s$ 
for $s\in [d^*]$. We first observe that according to the formulation of Algorithm \ref{alg:redCBCDBD}, the $v$th bit of $z_s$ 
with $v\in\{2,\dots,m-w_s\}$, has been selected by minimizing $h_{s,v,m,\bsgamma,\bsw}(z_{s,v-1}+z 2^{v-1})$ 
with respect to the choices $z\in\{0,1\}$, and that we have chosen $z_{s,v-1}$ by the same algorithm. 
By Lemma \ref{avg_sum} and Definition \ref{def:h_rv} this is equivalent to minimizing 
\[
 \frac{1}{2^{m-w_s-v}} \sum_{\overline{z}\in \ZZ_{2^{m-w_s-v}}} 
 H_{s,N,\bsgamma,\bsw} (z_1,\ldots,z_{s-1},z_{s,v-1} + z 2^{v-1}+ \overline{z}2^v)
\]
with respect to $z\in \{0,1\}$. By the standard averaging argument, this yields 
\begin{eqnarray}\notag
 \lefteqn{\argmin_{z\in  \{0,1\}} \frac{1}{2^{m-w_s-v}} \sum_{\overline{z}\in \ZZ_{2^{m-w_s-v}}} 
 H_{s,N,\bsgamma,\bsw} (z_1,\ldots,z_{s-1},z_{s,v-1} + z 2^{v-1}+ \overline{z}2^v)}\\\label{std_avg}
&\le & \frac{1}{2} \frac{1}{2^{m-w_s-v}} \sum_{z\in \ZZ_2}
\sum_{\overline{z}\in \ZZ_{2^{m-w_s-v}}} H_{s,N,\bsgamma,\bsw} (z_1,\ldots,z_{s-1},z_{s,v-1} + z 2^{v-1}+ \overline{z}2^v)\\\notag
&=& \frac{1}{2^{m-w_s-v +1}}\sum_{\overline{z}\in \ZZ_{2^{m-w_s-v +1}}} 
H_{s,N,\bsgamma,\bsw} (z_1,\ldots,z_{s-1},\underbrace{z_{s,v-2}+ \tilde{z} \, 2^{v-2}}_{=z_{s,v-1}} + \bar{z} \, 2^{v-1} ),
\end{eqnarray}
where we split up $z_{s,v-1}$ according to Algorithm \ref{alg:redCBCDBD} such that $\tilde{z}$ is the 
$(v-1)$th bit of $z_s$, selected in the previous step of the algorithm. Noting that the inequality in \eqref{std_avg}
holds for any $v\in\{2,\ldots,m-w_s\}$, we can inductively use this estimate for $v=m-w_s,m-w_s-1,\ldots, 2$ to obtain
\begin{eqnarray}\notag
 H_{s,N,\bsgamma,\bsw} (\bsz_{[s]}) &=& \argmin_{z\in\{0,1\}} H_{s,N,\bsgamma,\bsw} (z_1,\ldots,z_{s-1},z_{s,m-w_s-1}
 + z 2^{m-w_s -1})\\\label{ind_v_1}
 &\le& \frac{1}{2^{m-w_s-1}}\sum_{\overline{z}\in \ZZ_{2^{m-w_s -1}}} H_{s,N,\bsgamma,\bsw} (z_1,\ldots,z_{s-1}, 1 + \overline{z} 2 ),
\end{eqnarray}
where we used that $z_{s,1}=1$. Now setting $v=1$ in our expression in Lemma \ref{avg_sum}  
to equate the right-hand side term in \eqref{ind_v_1}, we get 
\begin{eqnarray*}
 H_{s,N,\bsgamma,\bsw} (\bsz_{[s]}) &\le& 
 \sum_{t=1}^{m-w_s} \frac{1}{2^{t-1}} \sum_{\ell=0}^{2^{w_s}-1} \sum_{\substack{k=\ell 2^t +1 \\  k \equiv 1 \tmod{2}}}^{(\ell +1)2^t -1}
 \sum_{\substack{\emptyset\neq\fraku \subseteq [s]\\ s\in\fraku}}\gamma_{\fraku}
 \prod_{\substack{j\in\fraku \\ j\neq s}} \log\left(\frac{1}{\sin^2 \left(\pi z_j \frac{k}{2^t} \frac{2^{m-w_s}}{2^{m-w_j}}\right)  }\right)\\
 &&\times \log\left(\frac{1}{\sin^2 (\pi k / 2)}\right)\\
 &&+ \sum_{r=1}^{s-1} \sum_{\substack{k=0\\ k\not\equiv 0 \tmod{2^{m-w_{r}}}}}^{N-1}
 \sum_{\substack{\emptyset\neq \fraku\subseteq [r]\\ r\in\fraku }} \gamma_{\fraku}
 \prod_{j\in\fraku} \log \left(\frac{1}{\sin^2 (\pi z_j k / 2^{m-w_j})}\right)\\
 &&+ \sum_{\ell=0}^{2^{w_s}-1} \ \ \sum_{t=1}^{m-w_s}\frac{2^{t-1}-1}{2^{t-1}} (\log 4) 
 \sum_{\substack{k=\ell 2^t + 1\\ k\equiv 1 \tmod{2}}}^{(\ell +1) 2^t-1}
 \sum_{\substack{\emptyset\neq \fraku\subseteq [s]\\ s\in\fraku }} \gamma_{\fraku}\\
 &&\times\prod_{\substack{j\in\fraku\\ j\neq s}} 
 \log \left(\frac{1}{\sin^2 \left(\pi z_j \frac{k}{2^t}\,\frac{2^{m-w_s}}{2^{m-w_j}}\right)}\right).
\end{eqnarray*}
Note that for odd $k$ we have $\log (\sin^{-2}(\pi k/2)) = \log 1 = 0$,
so we obtain, using \eqref{H_qty} and \eqref{eq:nice_identity},
\begin{eqnarray*}
 H_{s,N,\bsgamma,\bsw} (\bsz_{[s]}) &\le&
\sum_{r=1}^{s-1} \sum_{\substack{k=0\\ k\not\equiv 0 \tmod{2^{m-w_{r}}}}}^{N-1}
 \sum_{\substack{\emptyset\neq \fraku\subseteq [r]\\ r\in\fraku }} \gamma_{\fraku}
 \prod_{j\in\fraku} \log \left(\frac{1}{\sin^2 (\pi z_j k / 2^{m-w_j})}\right)\\
 &&+ \sum_{\ell=0}^{2^{w_s}-1} \ \ \sum_{t=1}^{m-w_s} (\log 4) 
 \sum_{\substack{k=\ell 2^t + 1\\ k\equiv 1 \tmod{2}}}^{(\ell +1) 2^t-1}
 \sum_{\substack{\emptyset\neq \fraku\subseteq [s]\\ s\in\fraku }} \gamma_{\fraku}\prod_{\substack{j\in\fraku\\ j\neq s}} 
 \log \left(\frac{1}{\sin^2 \left(\pi z_j \frac{k}{2^t}\,\frac{2^{m-w_s}}{2^{m-w_j}}\right)}\right)\\
 &=&   H_{s-1,N,\bsgamma,\bsw} (\bsz_{[s-1]})\\
 && + (\log 4) \sum_{\ell=0}^{2^{w_s}-1} \ \ \sum_{t=1}^{m-w_s}
 \sum_{\substack{k=\ell 2^t + 1\\ k\equiv 1 \tmod{2}}}^{(\ell +1) 2^t-1}
 \sum_{\substack{\emptyset\neq \fraku\subseteq [s]\\ s\in\fraku }} \gamma_{\fraku}\prod_{\substack{j\in\fraku\\ j\neq s}} 
 \log \left(\frac{1}{\sin^2 \left(\pi z_j \frac{k}{2^t}\,\frac{2^{m-w_s}}{2^{m-w_j}}\right)}\right)\\
 &=&   H_{s-1,N,\bsgamma,\bsw} (\bsz_{[s-1]})\\
 && + (\log 4)\sum_{\ell=0}^{2^{w_s}-1} 
 \sum_{k=\ell 2^{m-w_s} + 1}^{(\ell +1) 2^{m-w_s}-1}
 \sum_{\substack{\emptyset\neq \fraku\subseteq [s]\\ s\in\fraku }} \gamma_{\fraku}\prod_{\substack{j\in\fraku\\ j\neq s}} 
 \log \left(\frac{1}{\sin^2 \left(\pi z_j k / 2^{m-w_j}\right)}\right)\\
 &=&   H_{s-1,N,\bsgamma,\bsw} (\bsz_{[s-1]})\\
 && + (\log 4)\, \gamma_s 
 \sum_{\substack{k=0\\ k\not\equiv 0 \tmod{2^{m-w_s}}}}^{N-1}
 \sum_{\fraku\subseteq [s-1]} \gamma_{\fraku}\prod_{j\in\fraku} 
 \log \left(\frac{1}{\sin^2 \left(\pi z_j k / 2^{m-w_j}\right)}\right)\\
 &=&   H_{s-1,N,\bsgamma,\bsw} (\bsz_{[s-1]})\\
 && + (\log 4)\, \gamma_s \left[ 
 \sum_{\substack{k=0\\ k\not\equiv 0 \tmod{2^{m-w_s}}}}^{N-1} 1\right.\\ 
 &&+\left.
 \sum_{\substack{k=0\\ k\not\equiv 0 \tmod{2^{m-w_s}}}}^{N-1}
 \sum_{\emptyset\neq\fraku\subseteq [s-1]} \gamma_{\fraku}\prod_{j\in\fraku} 
 \log \left(\frac{1}{\sin^2 \left(\pi z_j k / 2^{m-w_j}\right)}\right)\right]\\
 &=&   H_{s-1,N,\bsgamma,\bsw} (\bsz_{[s-1]})\\
 && + (\log 4)\, \gamma_s \left[ (2^m - 2^{w_s})
 \phantom{\sum_{\substack{k=0\\ k\not\equiv 0 \tmod{2^{m-w_s}}}}^{N-1}} \right.\\ 
 &&+\left.
 \sum_{r=1}^{s-1}\sum_{\substack{k=0\\ k\not\equiv 0 \tmod{2^{m-w_s}}}}^{N-1}
 \sum_{\substack{\emptyset\neq\fraku\subseteq [r-1]\\ r\in\fraku}} \gamma_{\fraku}\prod_{j\in\fraku} 
 \log \left(\frac{1}{\sin^2 \left(\pi z_j k / 2^{m-w_j}\right)}\right)\right]\\ 
 &\le&   H_{s-1,N,\bsgamma,\bsw} (\bsz_{[s-1]})\\
 && + (\log 4)\, \gamma_s \left[ (2^m - 2^{w_s})
 \phantom{\sum_{\substack{k=0\\ k\not\equiv 0 \tmod{2^{m-w_s}}}}^{N-1}} \right.\\ 
 &&+\left.
 \sum_{r=1}^{s-1}\sum_{\substack{k=0\\ k\not\equiv 0 \tmod{2^{m-w_r}}}}^{N-1}
 \sum_{\substack{\emptyset\neq\fraku\subseteq [r-1]\\ r\in\fraku}} \gamma_{\fraku}\prod_{j\in\fraku} 
 \log \left(\frac{1}{\sin^2 \left(\pi z_j k / 2^{m-w_j}\right)}\right)\right]\\  
 &=&   H_{s-1,N,\bsgamma,\bsw} (\bsz_{[s-1]})\\
 &&+(\log4) \left[\gamma_s (2^m-2^{w_s}) + \gamma_s H_{s-1,N,\bsgamma,\bsw}(\bsz_{[s-1]})\right].
\end{eqnarray*} 
Hence we have the claimed result.
\end{proof}
Based on the result in Theorem \ref{theorem:H-induction} we can use an inductive argument to show 
that the quantity $H_{s,N,\bsgamma,\bsw}$ is sufficiently small if $\bsz_{[s]}$ has been constructed by Algorithm \ref{alg:redCBCDBD}. 
We obtain the following estimate.
\begin{theorem}\label{theorem:upper_bound_H}
Let $m\in \N$, $N = 2^m$, $\bsw = (w_j)_{j\ge 1}$ in $\N_0$ with $0=w_1\le w_2\le \cdots $,
and let  $\bsgamma = (\gamma_j)_{j\ge 1}$ be positive product weights. 
Furthermore, let the generating vector $\bsz=(Y_1z_1,\ldots,Y_dz_d) \in \Z^d$ be constructed by Algorithm \ref{alg:redCBCDBD}. Denote 
by $\bsz_{[s]}$ the vector $(z_1,\ldots,z_s)$ for $s\in [d]$.
Then the following upper bound on $H_{s,N,\bsgamma,\bsw}(\bsz_{[s]})$ holds for $s\in [d^*]$,
\[
H_{s,N,\bsgamma,\bsw}(\bsz_{[s]}) \le N \left[ -1 + \prod_{j=1}^s (1 + \gamma_j \log 4) \right].
\]
\end{theorem}
\begin{proof} Observe that due to Theorem \ref{theorem:H-induction},
\begin{eqnarray}
 H_{s,N,\bsgamma,\bsw}(\bsz_{[s]}) \le H_{s-1,N,\bsgamma,\bsw} (\bsz_{[s-1]})+ (\log 4)\left[ \gamma_s (2^m - 2^{w_s})
+  \gamma_s H_{s-1,N,\bsgamma,\bsw}(\bsz_{[s-1]})\right]
\end{eqnarray} 
holds for $2 \le s \le d^*$.
We can apply this estimate inductively to obtain the following,
\begin{eqnarray*}
 \lefteqn{H_{s,N,\bsgamma,\bsw} (\bsz_{[s]})}\\ 
 &\leq& (1+\gamma_s \log 4)\left[(1+\gamma_{s-1} \log 4) H_{s-2,N,\bsgamma,\bsw}(\bsz_{[s-2]}) 
 + \gamma_{s-1} (\log4) (2^m - 2^{w_{s-1}})\right] \\
 &&+ \gamma_s (\log4) (2^m - 2^{w_s})\\
 &=& H_{s-2,N,\bsgamma,\bsw}(\bsz_{[s-2]}) \prod_{j=s-1}^s (1+ \gamma_j \log 4) 
 + \gamma_{s-1} (\log4) (2^m - 2^{w_{s-1}})\\
 &&+\gamma_s \gamma_{s-1} (\log 4)^2 (2^m - 2^{w_{s-1}}) + \gamma_s (\log 4) (2^m - 2^{w_s})\\
 &\le & \left(\prod_{j=s-1}^s (1 + \gamma_j \log 4)\right)\left[(1+\gamma_{s-2} \log 4) H_{s-3,N,\bsgamma,\bsw} (\bsz_{[s-3]})
 + \gamma_{s-2} (\log 4) (2^m - 2^{w_{s-2}})\right]\\
 &&+ \gamma_{s-1} (\log 4) (2^m - 2^{w_{s-1}})+ \gamma_s \gamma_{s-1} (\log 4)^2 (2^m - 2^{w_{s-1}}) 
 + \gamma_s (\log 4) (2^m - 2^{w_s})\\
 &=& H_{s-3,N,\bsgamma,\bsw}(\bsz_{[s-3]}) \prod_{j=s-2}^s (1+ \gamma_j \log 4)
  + \gamma_{s-2} (\log 4) (2^m - 2^{w_{s-2}})\\
  &&+ \gamma_s \gamma_{s-2} (\log4)^2 (2^m - 2^{w_{s-2}})+  \gamma_{s-1} \gamma_{s-2} (\log 4)^2 (2^m - 2^{w_{s-2}})\\
  &&+ \gamma_s \gamma_{s-1} \gamma_{s-2} (\log4)^3 (2^m - 2^{w_{s-2}}) 
  +  \gamma_{s-1} (\log 4) (2^m - 2^{w_{s-1}})\\
 &&+ \gamma_s \gamma_{s-1} (\log 4)^2 (2^m - 2^{w_{s-1}}) + \gamma_s (\log 4) (2^m - 2^{w_s}).
\end{eqnarray*}
Repeating this argument inductively, we finally arrive at
\begin{eqnarray}\label{eq:bound_H_H_1}
 H_{s,N,\bsgamma,\bsw} (\bsz_{[s]}) &\le& 
 H_{1,N,\bsgamma,\bsw}(z_1) \prod_{j=2}^s (1+ \gamma_j \log 4)+ \sum_{\emptyset\neq \fraku \subseteq \{2:s\} } \gamma_{\fraku} (\log4)^{\abs{\fraku}} 
 (2^m - 2^{\min_{j\in\fraku} w_j})\nonumber\\
 &\le& 
 H_{1,N,\bsgamma,\bsw}(z_1) \prod_{j=2}^s (1+ \gamma_j \log 4)+ \sum_{\emptyset\neq \fraku \subseteq \{2:s\} } 
 \gamma_{\fraku} (\log4)^{\abs{\fraku}} 2^m. 
 \end{eqnarray}
However, recall that we have, by definition,
\begin{eqnarray*}
 H_{1,N,\bsgamma,\bsw}(z_1)&=& H_{1,N,\bsgamma,\bsw}(1)=\sum_{k=1}^{N-1}\gamma_1 u_1 (k)
 =\gamma_1 \sum_{k=1}^{N-1} \log \left(\frac{1}{\sin^2 (\pi k / 2^m)}\right)\\
 &=&-2\gamma_1 \sum_{k=1}^{N-1} \log \left(\sin\left(\frac{\pi k}{N}\right)\right)
 = \gamma_1 (N-m-1) \log 4\\ 
 &\le& \gamma_1 N \log 4 - \gamma_1 \log 4 \le \gamma_1 N \log 4,
\end{eqnarray*} 
where we used the identity 
\[
 \prod_{k=1}^{N-1}\left(2\sin\left(\frac{\pi k}{N}\right)\right) = N.
\]
Consequently, combining this bound with \eqref{eq:bound_H_H_1}, expanding the expression, and using 
that $\gamma_\setu =\prod_{j\in\setu}\gamma_j$, finally gives
\begin{eqnarray*}
 H_{s,N,\bsgamma,\bsw} (\bsz_{[s]}) &\le& 
 N \gamma_1 (\log 4) \left(\prod_{j=2}^s (1+\gamma_j \log 4 )\right) + N \sum_{\emptyset\neq \fraku \subseteq \{2:s\} } \gamma_{\fraku} 
 (\log 4)^{\abs{\fraku}} \\
 &=& N \left(\gamma_1 (\log 4) \prod_{j=2}^s (1+\gamma_j \log 4) + (-1) + \prod_{j=2}^s (1+\gamma_j \log 4)\right)\\
 &=& N\prod_{j=1}^s (1+ \gamma_j \log 4) - N  = N\left[ -1 + \prod_{j=1}^s(1+\gamma_j\log 4) \right],  
\end{eqnarray*}
which is the claim.

\end{proof}
We are now able to show the main result regarding the reduced component-by-component digit-by-digit construction.

\begin{theorem}\label{thm:optcoeff-dbd}
	Let $N=2^m$, with $m \in \N$, let $\bsw=(w_j)_{j\ge 1}$ be a sequence in $\N_0$ with $0=w_1\le w_2\le\cdots$, and let 
	$\bsgamma=(\gamma_j)_{j\ge1}$ be positive product weights. 
	Furthermore, denote by $\bsz = (Y_1z_1, \ldots, Y_dz_d)$ 
	the corresponding generating vector constructed by \RefAlg{alg:redCBCDBD}. Then the following estimate holds,
	\begin{equation} \label{eq:optcoeff-dbd}
		T_\bsgamma(N,\bsz,\bsw)
		\le
		\sum_{\emptyset\neq \fraku\subseteq [d^*] } \gamma_{\fraku} \frac{2 (6\log N)^{\abs{\fraku}+1}}{2^{m-\max_{j\in \fraku} w_j}}.
	\end{equation}
\end{theorem}

\begin{proof}

Recall from Theorem \ref{thm:T_target_CBCDBD} that 
\begin{eqnarray*}
 T_\bsgamma(N,\bsz,\bsw) &\le& \sum_{\emptyset\neq \fraku\subseteq [d^*] } \frac{\gamma_{\fraku}}{N}
 2^{w_{j_{\abs{\fraku}}}+1}  (6\log N)^{\abs{\fraku}} (1+\log N)\\
 &&+ \frac{1}{N} H_{d^*,N,\bsgamma,\bsw}
 - \sum_{\emptyset\neq \fraku\subseteq [d^*] } \gamma_{\fraku}(\log4)^{\abs{\fraku}}.
\end{eqnarray*}
Then, using the structure of product weights, along with the estimate for $H_{d^*,N,\bsgamma,\bsw}$ in Theorem \ref{theorem:upper_bound_H}, 
we obtain 
\begin{eqnarray*}
 T_\bsgamma(N,\bsz,\bsw) &\le& \frac{2(1+\log N)}{N}
 \sum_{\emptyset\neq \fraku\subseteq [d^*] } \gamma_{\fraku}
 2^{w_{j_{\abs{\fraku}}}}  (6\log N)^{\abs{\fraku}} \\
 &&+ \frac{1}{N} N\left(-1+\prod_{j=1}^{d^*} (1+\gamma_j \log 4)\right) - \prod_{j=1}^{d^*} (1+\gamma_j \log 4) +1\\
&=&\frac{2(1+\log N)}{N} \sum_{\emptyset\neq \fraku\subseteq [d^*] } \gamma_{\fraku}
 2^{w_{j_{\abs{\fraku}}}}  (6\log N)^{\abs{\fraku}} \\
&\le & 
 \sum_{\emptyset\neq \fraku\subseteq [d^*] } \gamma_{\fraku} \frac{2 (6\log N)^{\abs{\fraku}+1}}{2^{m-\max_{j\in \fraku} w_j}},
\end{eqnarray*}
where we used that for a non-empty set $\fraku\subseteq [d^*]$ we have $w_{j_{\abs{\fraku}}}=\max_{j\in\fraku} w_j$. This is the claimed result. 
\end{proof}

We now have the following corollary.
\begin{corollary} \label{cor:main-result-red_dbd}
	Let $N=2^m$, with $m \in \N$, let $\bsw=(w_j)_{j\ge 1}$ be a sequence in $\N_0$ with $0=w_1\le w_2\le\cdots$, 
	and let $\bsgamma=(\gamma_j)_{j\ge1}$ be positive product weights satisfying
	\begin{align*}
		\sum_{j \ge 1} \gamma_j2^{w_j}
		&<
		\infty
		.
	\end{align*}
Furthermore, denote by $\bsz=(Y_1 z_1,\ldots,Y_d z_d)$ the generating vector constructed by \RefAlg{alg:redCBCDBD} run for the weights 
$\bsgamma = (\gamma_j)_{j\ge1}$. Then, for any $\delta>0$ and each $\alpha>1$, the generating vector $\bsz$ satisfies
	\begin{equation*}
		e_{N,d,\alpha,\bsgamma^{\alpha}}(\bsz)
		\le
		\frac{1}{N^\alpha}\left(C_1(\bsgamma^\alpha) + C_2(\bsgamma,\delta)N^{\alpha\delta}\right)
	\end{equation*}  
with weight sequence $\bsgamma^{\alpha}=(\gamma_j^{\alpha})_{j\ge 1}$ and positive constants 
$C_1(\bsgamma^\alpha)$ and $C_2(\bsgamma,\delta)$, which are independent of $d$ and $N$. 
Additionally, if Algorithm \ref{alg:redCBCDBD} is run for weights $\bsgamma^{1/\alpha}=(\gamma_j^{1/\alpha})_{j\ge 1}$ with $\alpha >1$, satisfying 
\begin{align*}
		\sum_{j \ge 1} \gamma_j^{1/\alpha}2^{w_j}
		&<
		\infty,
\end{align*}
then, for any $\delta >0$, the resulting generating vector $\widetilde{\bsz}=(Y_1\widetilde{z}_1,\ldots,Y_d \widetilde{z}_d)$ satisfies the error bound
\begin{equation*}
		e_{N,d,\alpha,\bsgamma}(\widetilde{\bsz})
		\le
		\frac{1}{N^\alpha}\left(F_1(\bsgamma) + F_2(\bsgamma^{1/\alpha},\delta)N^{\alpha\delta}\right),
	\end{equation*}  with positive constants $F_1(\bsgamma)$ and $F_2(\bsgamma^{1/\alpha},\delta)$, which are independent of $d$ and $N$.
\end{corollary}
\begin{proof}
We know from Proposition~\ref{prop:trunc_error} that the worst-case error $e_{N,d,\alpha,\bsgamma^\alpha}(\bsz)$ satisfies
	\begin{align*}
		e_{N,d,\alpha,\bsgamma^\alpha}(\bsz) 
	    \le
	\sum_{\emptyset\neq \setu \subseteq [d]} \gamma_\setu^\alpha \, \frac{(4\zeta (\alpha))^{\abs{\setu}}}{2^{\alpha\max(0,m-\max_{j\in\setu}w_j)}}
	+ T_{\alpha,\bsgamma^\alpha}(N,\bsz,\bsw).
	\end{align*} al
Taking into account that we have product weights $\gamma_\fraku = \prod_{j\in\fraku} \gamma_j$, this yields
    \begin{align*}
	  	e_{N,d,\alpha,\bsgamma^\alpha}(\bsz) &\le \frac{1}{N^{\alpha}} 
	  	\sum_{\emptyset\neq \setu \subseteq [d]} \left(\prod_{j\in\fraku} (\gamma_j^\alpha \, 
	  	4\zeta (\alpha))\right)2^{\alpha\max_{j\in\setu}w_j} + T_{\alpha,\bsgamma^\alpha}(N,\bsz,\bsw)\\
	  	&\le \frac{1}{N^{\alpha}}  \prod_{j=1}^d\left(1+\gamma_j^\alpha \, 4\zeta (\alpha)2^{\alpha w_j}\right) + T_{\alpha,\bsgamma^\alpha}(N,\bsz,\bsw),
	\end{align*}
where we used that for $\emptyset \neq \fraku \subseteq [d]$ we have $\max_{j\in\setu}w_j \leq \sum_{j\in\setu}w_j$. 
Since $\alpha >1$, we can use an inequality sometimes called Jensen's inequality, which states that 
$\sum_{i=1}^n a_i \leq \left(\sum_{i = 1}^n a_i^p\right)^{1/p}$ for non-negative $a_1, \dots, a_n$ and $0\le p \le 1$, 
and thus we have 
\[
T_{\alpha,\bsgamma^\alpha}(N,\bsz,\bsw) = \sum_{\bszero\neq\bsell \in M_{N,d,\bsw}}\frac{\delta_N(\bsell\cdot\bsz)}{r_{\alpha,\bsgamma^\alpha}(\bsell)} 
\le \left(\sum_{\bszero\neq\bsell \in M_{N,d,\bsw}}\frac{\delta_N(\bsell\cdot\bsz)}{r_{1,\bsgamma}(\bsell)}\right)^\alpha =(T_\bsgamma(N,\bsz,\bsw))^\alpha,
\] 
and by Theorem \ref{thm:optcoeff-dbd} we know that the $\bsz$ generated by Algorithm \ref{alg:redCBCDBD} is such that
	\begin{equation*}
	    T_\bsgamma(N,\bsz,\bsw) \le \sum_{\emptyset\neq \fraku\subseteq [d^*] } \gamma_{\fraku} \frac{2 (6\log N)^{\abs{\fraku}+1}}{2^{m-\max_{j\in \fraku} w_j}}.
	\end{equation*}
From this, we deduce, using $\gamma_\fraku = \prod_{j\in\fraku} \gamma_j$, that we have 
	\begin{eqnarray}\label{eq:6logN}
	    2^mT_\bsgamma(N,\bsz,\bsw) &\le& 2(6\log N)\sum_{\emptyset\neq \fraku\subseteq [d^*] } \gamma_{\fraku} (6\log N)^{\abs{\fraku}}2^{\max_{j\in \fraku} w_j}\nonumber\\
	    &\le& 2(6\log N) \sum_{\emptyset\neq \fraku\subseteq [d^*] } 
	    \left(\prod_{j\in\fraku} (\gamma_j\, 6\log N)\right)\, 2^{\sum_{j\in\fraku}w_j}\nonumber\\
	    &\le& 2(6\log N) \prod_{j=1}^{d^*}\left(1 + \gamma_j\,(6\log N)\, 2^{w_j} \right)\nonumber\\ 
	    &\le& (1+2(6\log N)) \prod_{j=1}^{d^*}\left(1 + \gamma_j\,(6\log N)\, 2^{w_j} \right)\nonumber\\
	    &\le& \widetilde{C}(\delta/2)2^{m\delta/2}\prod_{j=1}^{\infty}\left(1 + \gamma_j\,(6\log N)\, 2^{w_j} \right),
	\end{eqnarray}
for arbitrary $\delta> 0$, where $\widetilde{C}(\delta/2)$ is a constant depending only on $\delta$, 
and where we again used for $\emptyset \neq \fraku \subseteq [d^*]$ that we have 
$\max_{j\in\setu}w_j \leq \sum_{j\in\setu}w_j$. Due to the given condition imposed on the weights, i.e., 
$\sum_{j\ge1}\gamma_j2^{w_j}<\infty$, we can use the result in \cite[Lemma 3]{HN03} to see that  the product in \eqref{eq:6logN} can be bounded by 
$\hat{C}(\bsgamma,\delta)2^{m\delta}$, where $\hat{C}(\bsgamma,\delta)$ may depend on the weights 
$\bsgamma$ and $\delta$, but is independent of the dimension. This yields
	\begin{equation*}
	    (T_\bsgamma(N,\bsz,\bsw))^\alpha \le \frac{1}{N^\alpha}\left( \widetilde{C}(\delta/2) \right)^\alpha\left(\hat{C}(\bsgamma,\delta)\right)^\alpha N^{\alpha\delta} .
	\end{equation*}
Setting then $C_1(\bsgamma^\alpha)= \prod_{j=1}^\infty(1+\gamma_j^{\alpha} 4\zeta(\alpha)2^{\alpha w_j})$, which is finite due to our 
assumption on the weights, and $C_2(\bsgamma,\delta) = \left(\widetilde{C}(\delta/2)\right)^\alpha\left(\hat{C}(\bsgamma,\delta) \right)^\alpha$ 
we get the claimed estimate.

Similarly, for weights $\bsgamma^{1/\alpha}$ and due to the condition imposed on the weights, i.e., 
$\sum_{j\geq 1}\gamma_j^{1/\alpha}2^{w_j} < \infty$, and by using the result in \cite[Lemma 3]{HN03} we have
\begin{equation*}
    (T_{\bsgamma^{1/\alpha}}(N,\widetilde{\bsz},\bsw))^\alpha \le 
    \frac{1}{N^\alpha}\left( \widetilde{C}(\delta/2) \right)^\alpha\left(\hat{C}(\bsgamma^{1/\alpha},\delta)\right)^\alpha N^{\alpha\delta}.
\end{equation*}
Then setting $F_1(\bsgamma)= \prod_{j=1}^\infty(1+\gamma_j 4\zeta(\alpha)2^{w_j})$ and 
$F_2(\bsgamma^{1/\alpha},\delta) = \left(\widetilde{C}(\delta/2)\right)^\alpha\left(\hat{C}(\bsgamma^{1/\alpha},\delta) \right)^\alpha$ we get the claimed estimate.
\end{proof}	
The result in Corollary \ref{cor:main-result-red_dbd} involves two cases regarding the worst-case error 
behavior of the generating vectors constructed by Algorithm \ref{alg:redCBCDBD}. 
To be more precise, we can run the algorithm with weights $\bsgamma$, and hence it does not depend on the parameter $\alpha$, 
then the algorithm produces a generating vector for which bounds on the worst-case errors in the spaces $E_{d,\bsgamma^\alpha}^\alpha$ 
hold simultaneously for all $\alpha >1$. On the other hand, when we run Algorithm \ref{alg:redCBCDBD} with weights $\gamma^{1/\alpha}$, 
we have a dependence on the parameter $\alpha$, and the algorithm yields error bounds for the worst-case error in the spaces $E_{d,\bsgamma}^\alpha$.

\section{Fast implementation of the construction scheme} \label{sec:fast_impl}

In this section we discuss the efficient implementation of the introduced reduced CBC-DBD algorithm and analyze its complexity.

\subsection{Implementation and cost analysis of the reduced CBC-DBD algorithm}
We recall the definition of the quality function in Definition \ref{def:h_rv}, 
and we see that for product weights $h_{d,v,m,\bsgamma,\bsw}$ can be rewritten as follows.
Let $x \in \N$ be an odd integer, let $m, d \in \N$, and let $\bsgamma = (\gamma_j)_{j\ge 1} $ be positive product weights. 
For $1 \le s \le d$, $1 \le v \le m-w_s$, and odd integers $z_1,\ldots,z_{s-1}$ the quality function $h_{s,v,m,\bsgamma,\bsw}$ reads

\begin{align*}\label{eq:fast-eval-h_sw}
    \!\!\!\!\!\!h_{s,v,m,\bsgamma,\bsw} (x):&= \sum_{t=v}^{m-w_s} \frac{1}{2^{t-v}} 
 \sum_{\ell=0}^{2^{w_s}-1} \ \ \sum_{\substack{k=\ell 2^t + 1\\ k\equiv 1 \tmod{2}}}^{(\ell +1) 2^t-1}\left(1+ 
 \gamma_s\log\left(\frac{1}{\sin^2(\pi kx/2^v)} \right) \right) \\ 
 &\times \prod_{j=1}^{s-1}\left[1 + \gamma_j\log \left(\frac{1}{\sin^2 \left(\pi z_j \frac{k}{2^t}\,\frac{2^{m-w_s}}{2^{m-w_j}}\right)}\right) \right],
\end{align*}
where the components $z_1,\dots,z_{s-1}$ have been determined in the previous steps of the algorithm. We are interested in the cost of a single evaluation of the function $h_{s,v,m,\bsgamma,\bsw}$, which is crucial for the total cost of Algorithm 1, hence we will discuss an efficient evaluation procedure in the following paragraph.

For integers $t \in \{2,\ldots,m-w_s\}$ (note that we always have $v\ge 2$ in Algorithm \ref{alg:redCBCDBD}, so we do not need to 
consider the case $t=1$) and odd $k \in \{\ell2^t +1,\ldots,(\ell+1)2^t - 1\}$ for $\ell\in\{0,\ldots,2^{w_s}-1\}$, we define $r(s, t, k)$
as
\begin{equation*}
	r(s,t,k)
	:=
	\prod_{j=1}^s \left(1+ \gamma_j \log\left(\frac{1}{\sin^2\left(\pi z_j\frac{k}{2^t}\frac{2^{m-w_s}}{2^{m-w_j}}\right)} \right) \right),
\end{equation*}
and observe that for the evaluation of $h_{s,v,m,\bsgamma,\bsw}(x)$ we can compute and store the terms $r(s-1,t,k)$ 
for suitable values of $t$ since they are independent of $v$ and $x$. This way we can rewrite $h_{s,v,m,\bsgamma,\bsw}(x)$ as
\begin{equation} \label{eq:fast-eval-h_sw}
	\sum_{t=v}^{m-w_s} \frac{1}{2^{t-v}} \sum_{\ell=0}^{2^{w_s}-1} \ \ \sum_{\substack{k=\ell 2^t + 1\\ k\equiv 1 \tmod{2}}}^{(\ell +1) 2^t-1}
	r(s-1,t-w_{s-1}+w_{s},k) \left( 1 + \gamma_s\log\left(\frac{1}{\sin^2\left(\pi kx/2^v\right)} \right) \right)
	.
\end{equation}
Note that if $v\le t\le m-w_s$, then $0\le v-w_{s-1}+w_s \le t - w_{s-1} + w_s \le m- w_{s-1}$.
In Algorithm \ref{alg:redCBCDBD}, after having determined $z_{s}$, the values of $r(s,v,k)$ for odd integers 
${k\in\{\ell2^v+1,\ldots,(\ell+1)2^v-1\}}$ with $\ell\in\{0,\ldots,2^{w_s}-1\}$ are computed via the recurrence relation
\begin{equation*}
	r(s,t,k) 
	=
	r(s-1,t-w_{s-1}+w_{s},k) \left( 1 + \gamma_s\log\left(\frac{1}{\sin^2(\pi kz_{s}/2^t)} \right) \right)
	.
\end{equation*}

Now, for an algorithmic implementation, we introduce the vector $\bsu=(u(1),\ldots,u(2^{m}-1)) \in \R^{2^{m}-1}$ , 
whose components, for the current $s \in \{1,\ldots,d^*\}$, are given by 
\begin{equation*}
	u(k \, 2^{m-w_s-t}) 
	= 
	\prod_{j=1}^s \left( 1 + \gamma_j \log\left(\frac{1}{\sin^2\left(\pi z_j \frac{k}{2^t}\frac{2^{m-w_s}}{2^{m-w_j}}\right)}\right) \right)
\end{equation*}
for each $t\in\{1,\ldots,m-w_s\}$ and corresponding odd index $k\in\{\ell2^t+1,\ldots,(\ell+1)2^t - 1\}$ for $\ell\in\{0,\ldots,2^{w_s}-1\}$. 
Note that $k \, 2^{m-w_s-t}$ runs through the whole range $\{1,\ldots,2^m-1\}$ when $t$, $k$, and $\ell$ are chosen as stated.
Furthermore, note that the quantity $u(k \, 2^{m-w_s-t}) $ corresponds to $r(s,t,k)$ for $t\ge 2$, and that
for the evaluation of $h_{s,v,m,\bsgamma,\bsw}$ we do not 
require the values of $r(s,t,k)$ for $t=2,\ldots,v-1$. Combining these findings leads to the following fast implementation of Algorithm \ref{alg:redCBCDBD}.

\begin{algorithm}[H]
	\caption{Fast reduced component-by-component digit-by-digit algorithm}
	\label{alg:fast-redcbcdbd}
	\vspace{5pt}
	\textbf{Input:} Integers $m, d \in \N$, positive product weights $\bsgamma=(\gamma_j)_{j\ge 1}$, 
	and integer reduction indices $(w_j)_{j\ge 1}$ with $0=w_1\le w_2\le \cdots$. \\
	\vspace{-10pt}
	\begin{algorithmic}
		\FOR{$t=2$ \TO $m =m-w_1$}
		\FOR{$k =1$ \TO $2^t -1$\textbf{ in steps of} $2$}
		\STATE $u(k2^{m-w_1-t}) =  \left( 1+\gamma_1\log\left(\frac{1}{\sin^2(\pi k/2^t)} \right)\right)$
		\ENDFOR
		\ENDFOR
		\vspace{5pt}
		\STATE Set $s =1$ and let $d^*$ be the largest integer such that $w_{d^*} < m$. \\ Set $z_{1,1} = \cdots = z_{d,1} = 1$.\\
		\vspace{5pt}
		\WHILE{$s\leq\min\{d,d^*\}$}
		\vspace{3pt}
		\FOR{$v=2$ \TO $m-w_s$}
		\STATE $z^{\ast} = \underset{z \in \{0,1\}}{\argmin} \;  h_{s,v,m,\bsgamma,\bsw}(z_{s,v-1} + z \,2^{v-1})$ 
		with $h_{s,v,m,\bsgamma,\bsw}$ evaluated using  \eqref{eq:fast-eval-h_sw}
		\STATE $z_{s,v} = z_{s,v-1} + z^{\ast} 2^{v-1}$
		\FOR{$\ell = 0$ \TO $2^{w_s}-1$}
		\FOR{$k = \ell 2^v +1$ \TO $(\ell+1)2^v -1$\textbf{ in steps of} $2$}
		\STATE $u(k2^{m-w_s-v}) = u(k2^{m-w_{s-1}-v}) \left( 1 + \gamma_s \log\left(\frac{1}{\sin^2(\pi k z_{s,v}/2^v)} \right) \right)$
		\ENDFOR
		\ENDFOR
		\ENDFOR
		\STATE Set $z_s := z_{s,m-w_s}$.
		\ENDWHILE
		\IF{$d>d^*$}
		\STATE set $z_{d^*+1} = \cdots = z_{d} =1$
		\ENDIF
		\vspace{5pt}
		\STATE Set $\bsz = (Y_1z_1,\ldots,Y_dz_d)$.
	\end{algorithmic}
	\vspace{5pt}
	\textbf{Return:} Generating vector $\bsz = (Y_1z_1,\ldots,Y_dz_d)$ for $N=2^m$.
\end{algorithm} The computational complexity of Algorithm \ref{alg:fast-redcbcdbd} is then summarized in the following theorem.

\begin{theorem}\label{thm:cost-redcbcdbd}
    Let $N = 2^m$ with $m \in \N$, let $\bsgamma = (\gamma_j)_{j\ge 1}$ be a given sequence of positive weights, 
    and let integer reduction indices $w_j$ with $0 = w_1 \le w_2 \le \cdots $ be given. Moreover, denote by $d^*$ the largest integer such that
    $w_{d^*} < m$. Then Algorithm \ref{alg:fast-redcbcdbd} constructs a generating  
	vector $\bsz=(z_1,\ldots,z_d)$ using 
	$$
	\calO\left(\min\{d,d^*\}\, 2^m+ \sum_{s=1}^{\min\{d,d^*\}}(m-w_s) 2^{m}\right)
	$$ 
	operations and requiring $\calO(2^m)$ memory.

\end{theorem}   

\begin{proof} 
	Due to the relation in \eqref{eq:fast-eval-h_sw}, the cost of evaluating $h_{s,v,m,\bsgamma,\bsw}(x)$ for one fixed $v$ can be reduced to 
	$\calO(\sum_{t=v}^{m-w_s} 2^{w_s +  t-1})$ operations. Moreover, updating the values of $u$ needs $\calO (2^{w_s + v-1})$ 
	operations, so the computational cost for one fixed $v$ in the inner loop over $v = 2,\dots,m-w_s$ of Algorithm \ref{alg:fast-redcbcdbd} is 
	of order $\calO(\sum_{t=v}^{m-w_s} 2^{w_s +  t-1})$.
	Thus, the number of calculations in the inner loop over $v = 2,\dots,m-w_s$ of Algorithm \ref{alg:fast-redcbcdbd} is of order
	\begin{align*}
	\calO\left(\sum_{v=2}^{m-w_s} 2 \sum_{t=v}^{m-w_s} 2^{w_s + t-1}\right) 
	&=
	\calO\left(\sum_{v=2}^{m-w_s} \sum_{t=v}^{m-w_s} 2^{w_s + t} \right)
	=
	\calO\left((m-w_s) \, 2^{m+1} - 2(2^{m+1} - 2^{w_s +1}) \right)\\
	&=
	\calO\left((m-w_s) \, 2^{m} \right)
	.
	\end{align*}
	Hence, the outer loop over $s=1,\ldots,\min \{d,d^*\}$, which is the main cost of Algorithm \ref{alg:fast-redcbcdbd}, can be executed in 
	$$
	\calO\left(\sum_{s=1}^{\min\{d,d^*\}}(m-w_s) 2^{m} \right)
	$$ 
	operations.	Furthermore, we observe that initialization and updating of the vector $\bsu \in \R^{2^m-1}$ can both be executed in 
	$\calO(\min\{d,d^*\}2^m)$ operations. 
	Additionally, storing the vector $\bsu$ requires $\calO(2^m)$ of memory.
\end{proof}

The run-time of Algorithm \ref{alg:fast-redcbcdbd} can also be reduced further by precomputing and storing the $2^m$ values 
\begin{equation*}
\log\left(\frac{1}{\sin^2(\pi k/2^m)}  \right)
\quad \text{for} \quad
k=1,\ldots,2^m-1.
\end{equation*}

The derivation leading to the fast implementation in Algorithm \ref{alg:fast-redcbcdbd} is using arguments that were also used in \cite{EKNO21}, where
a non-reduced component-by-component digit-by-digit construction for lattice rules in weighted Korobov spaces has been studied.

\section{Numerical results} \label{sec:num}
In this section, we illustrate the error convergence behavior of the lattice rules constructed by the reduced fast CBC-DBD algorithm and display the computational complexity of the construction using numerical experiments. As in the previous section, we consider lattice rules in the weighted Korobov space $E_{s,\bsgamma}^\alpha$ of smoothness $\alpha >1$, and we assume product weights $\gamma_\setu = \prod_{j \in \setu} \gamma_j$ given in terms of positive reals $(\gamma_j)_{j \ge 1}$. For $\bsz = (z_1,\dots,z_d)$, the worst-case error is then given by 
\begin{equation*}
	e_{N,d,\alpha,\bsgamma}(\bsz)
	=
	-1 + \frac{1}{2^m} \sum_{k=0}^{2^m-1} \prod_{j=1}^d \left(1 + \gamma_j \,\sum_{\ell \in \Z_*}\frac{e^{2\pi \icomp k\ell z_j/2^m}}{\abs{\ell}^\alpha} \right).
\end{equation*}
To demonstrate the performance of the algorithm, we compare the worst-case errors of the constructed lattice rules as well as the algorithm's computation times with the corresponding quantities obtained by the non-reduced component-by-component digit-by-digit algorithm, see \cite{EKNO21}. Both constructions deliver lattice rules for the space $E_{s,\bsgamma}^\alpha$ consisting of $2^m$ cubature points. The different algorithms have been implemented in double-precision and arbitrary-precision floating-point arithmetic, with the latter provided by the multi-precision Python library mpmath.

\subsection{Error convergence behavior}
Let $m,d \in \N$, $\alpha > 1$, a sequence of positive weights $\bsgamma = (\gamma_j)_{j \ge 1}$, and reduction indices $(w_j)_{j\geq 1}$ 
with $0=w_1\le w_2 \le \cdots$ be given. 
In particular, we consider the convergence rate of the worst-case error $e_{2^m,d,\alpha,\bsgamma^\alpha}(\bsz)$ for different weight sequences $\bsgamma = (\gamma_j)_{j\geq1}$ of the form $\gamma_j = c^j$ with $c \in (0,1)$ or $\gamma_j = 1/j^q$ with $q >1$ and reduction indices of the form $w_j = \lfloor p\log_2 j\rfloor$ with $p>0$. We display the computational results for dimension $d=100$ for different sequences of product weights, different values of $m$, and reduction indices $w_j$. We stress that the almost optimal error rates of $\calO(N^{-\alpha+\delta})$, as guaranteed by Corollary \ref{cor:main-result-red_dbd}, may not always be visible for the weights, reduction indices, and ranges of $N$ considered in our numerical experiments. The graphs shown are therefore to be understood as illustrations of the pre-asymptotic behavior of the worst-case error.
Figures \ref{fig:red_cbc_dbd_1} and \ref{fig:red_cbc_dbd_2} show numerical results using different choices of weights $(\gamma_j)_{j\ge 1}$ 
for reduction indices of the form $w_j = \lfloor 2 \log_2 j\rfloor$ and $w_j=\lfloor \frac{7}{2}\log_2 j\rfloor$, respectively. The generating vectors $\bsz$ are obtained by the reduced fast CBC-DBD algorithm and the non-reduced fast CBC-DBD algorithm, respectively. 

\begin{figure}
	\centering
	\textbf{Error convergence in the space $E_{d,\bsgamma}^{\alpha}$ with $d=100, \alpha=2$, $w_j = \lfloor2\log_2 j\rfloor$.} \par\medskip 
	\hspace{-0.25cm}
	\centering
	\begin{subfigure}[b]{0.5\textwidth}
		\centering
		\begin{tikzpicture}
		\begin{axis}[%
width=0.811\textwidth,
height=0.8\textwidth,
at={(0\textwidth,0\textwidth)},
scale only axis,
xmode=log,
xmin=10,
xmax=300000,
xminorticks=true,
xlabel style={font=\color{white!15!black}},
xlabel={Number of points $N=2^m$},
ymode=log,
ymin=1e-10,
ymax=0.01,
yminorticks=true,
ylabel style={font=\color{white!15!black}},
ylabel={Worst-case error $e_{N,d,\alpha,\mathbf{\gamma^{\alpha}}}(\mathbf{z})$},
axis background/.style={fill=white},
xmajorgrids,
ymajorgrids,
legend style={at={(0.03,0.03)}, anchor=south west, legend cell align=left, align=left, draw=white!15!black}
]
		\addplot [color=mycolor1-fig,solid, line width=0.6pt, mark=o, mark options={solid}, forget plot]
  table[row sep=crcr]{%
64	0.00136397342871476\\
128	0.00041495206968083\\
256	0.000113777898495654\\
512	3.57186608991275e-05\\
1024	8.96735413484793e-06\\
2048	3.12393880750252e-06\\
4096	5.94557183683682e-07\\
8192	1.73173640260874e-07\\
16384	4.81706937364865e-08\\
32768	1.07434643818022e-08\\
65536	3.40729981226737e-09\\
131072	1.28708617753193e-09\\
};
\addplot [color=mycolor2-fig,solid,line width=0.6pt,mark=triangle,mark options={solid},forget plot]
		table[row sep=crcr]{%
64	0.00338976825721932\\
128	0.000776435190914774\\
256	0.000249289329498733\\
512	8.79701854432796e-05\\
1024	4.56201064529152e-05\\
2048	4.28360533608438e-06\\
4096	2.81375393339083e-06\\
8192	6.77676918354019e-07\\
16384	1.68694186004295e-07\\
32768	2.36560764606151e-08\\
65536	1.08732055074483e-08\\
131072	9.1273289066741e-09\\
};
	\addplot [color=mycolor3-fig,dashed,line width=0.8pt]
		table[row sep=crcr]{%
64	0.00186437254147063\\
128	0.000498446504017713\\
256	0.000133261411998431\\
512	3.5627903465413e-05\\
1024	9.52524430219703e-06\\
2048	2.5466072990968e-06\\
4096	6.80844346881189e-07\\
8192	1.82026111699467e-07\\
16384	4.86653160773161e-08\\
32768	1.30108420533383e-08\\
65536	3.47849401960058e-09\\
131072	9.29987513090472e-10\\
};
\addlegendentry{$\calO (N^{-1.90})$}
		\end{axis}
		\end{tikzpicture}
		\caption{Weight sequence $\bsgamma=(\gamma_j)_{j=1}^d$ with $\gamma_j = 1/j^3$.}    
	\end{subfigure}
	\begin{subfigure}[b]{0.5\textwidth}  
		\centering 
		\begin{tikzpicture}
		\begin{axis}[%
width=0.811\textwidth,
height=0.8\textwidth,
at={(0\textwidth,0\textwidth)},
scale only axis,
xmode=log,
xmin=20,
xmax=200000,
xminorticks=true,
xlabel style={font=\color{white!15!black}},
xlabel={Number of points $N=2^m$},
ymode=log,
ymin=1e-10,
ymax=0.01,
yminorticks=true,
ylabel style={font=\color{white!15!black}},
ylabel={Worst-case error $e_{N,d,\alpha,\mathbf{\gamma^{\alpha}}}(\mathbf{z})$},
axis background/.style={fill=white},
xmajorgrids,
ymajorgrids,
legend style={at={(0.03,0.03)}, anchor=south west, legend cell align=left, align=left, draw=white!15!black}
]
		
		\addplot [color=mycolor1-fig,solid, line width=0.6pt, mark=o, mark options={solid}, forget plot]
  table[row sep=crcr]{%
64	0.000803660679099214\\
128	0.000200971272724395\\
256	5.02498739194294e-05\\
512	1.25700964904643e-05\\
1024	3.14241304450872e-06\\
2048	7.86290904231181e-07\\
4096	1.96387468195797e-07\\
8192	4.9114858240217e-08\\
16384	1.22747998201088e-08\\
32768	3.06929454004206e-09\\
65536	7.68023623822221e-10\\
131072	1.92351367445007e-10\\
};
\addplot [color=mycolor2-fig,solid,line width=0.6pt,mark=triangle,mark options={solid},forget plot]
		table[row sep=crcr]{%
64	0.000804842338171632\\
128	0.000201124614944839\\
256	5.02501348906322e-05\\
512	1.25752463046481e-05\\
1024	3.1589729322081e-06\\
2048	7.86229217616706e-07\\
4096	1.97649721241715e-07\\
8192	4.92347799680876e-08\\
16384	1.23340263398877e-08\\
32768	3.06931547528559e-09\\
65536	7.68771535002009e-10\\
131072	1.93759692741932e-10\\
};
		
		\addplot [color=mycolor3-fig,dashed,line width=0.8pt]
		table[row sep=crcr]{%
64	0.000442663285994398\\
128	0.000110662158896541\\
256	2.76646241039251e-05\\
512	6.9159271285046e-06\\
1024	1.72892455965088e-06\\
2048	4.32216834767362e-07\\
4096	1.08050632523862e-07\\
8192	2.70117641185601e-08\\
16384	6.75271753393575e-09\\
32768	1.68812351140708e-09\\
65536	4.2201691029484e-10\\
131072	1.05500735800045e-10\\
};
\addlegendentry{$\calO(N^{-2.00})$}
		\end{axis}
		\end{tikzpicture}
		\caption{Weight sequence $\bsgamma=(\gamma_j)_{j=1}^d$ with $\gamma_j = 1/j^8$.}
	\end{subfigure}
	\vskip\baselineskip
	\hspace{-0.25cm}
	\centering
	\begin{subfigure}[b]{0.5\textwidth}
		\centering
		\begin{tikzpicture}
		
		\begin{axis}[%
width=0.811\textwidth,
height=0.8\textwidth,
at={(0\textwidth,0\textwidth)},
scale only axis,
xmode=log,
xmin=10,
xmax=500000,
xminorticks=true,
xlabel style={font=\color{white!15!black}},
xlabel={Number of points $N=2^m$},
ymode=log,
ymin=1e-11,
ymax=1e-3,
yminorticks=true,
ylabel style={font=\color{white!15!black}},
ylabel={Worst-case error $e_{N,d,\alpha,\mathbf{\gamma^{\alpha}}}(\mathbf{z})$},
axis background/.style={fill=white},
xmajorgrids,
ymajorgrids,
legend style={at={(0.03,0.03)}, anchor=south west, legend cell align=left, align=left, draw=white!15!black}
]
	
	\addplot [color=mycolor1-fig, line width=0.6pt, mark=o, mark options={solid}, forget plot]
  table[row sep=crcr]{%
64	0.000104240109217907\\
128	2.9195489069013e-05\\
256	7.72906106593819e-06\\
512	2.30081487959806e-06\\
1024	5.79809038743061e-07\\
2048	1.81503073600323e-07\\
4096	3.66692191473153e-08\\
8192	1.03498552821969e-08\\
16384	2.3686872108615e-09\\
32768	6.82853999181353e-10\\
65536	2.41399001025539e-10\\
131072	6.72599684960036e-11\\
};

		\addplot [color=mycolor2-fig,solid,line width=0.6pt,mark=triangle,mark options={solid},forget plot]
		table[row sep=crcr]{%
	64	0.000179207417196795\\
128	4.22472342413148e-05\\
256	1.33937975609161e-05\\
512	2.70075749518944e-06\\
1024	2.21288899060296e-06\\
2048	1.84718767091478e-07\\
4096	1.01959628037341e-07\\
8192	1.79232462636359e-08\\
16384	6.22338458861278e-09\\
32768	7.86747388915077e-10\\
65536	3.5181537062487e-10\\
131072	2.28203067086458e-10\\
};

		\addplot [color=mycolor3-fig,dashed,line width=0.8pt]
		table[row sep=crcr]{%
		64	0.000116484821177917\\
128	2.95793933569213e-05\\
256	7.51119761798928e-06\\
512	1.90734437909919e-06\\
1024	4.84338552319325e-07\\
2048	1.2298976306186e-07\\
4096	3.1231215738613e-08\\
8192	7.93065058610776e-09\\
16384	2.01385752144032e-09\\
32768	5.113858027948e-10\\
65536	1.29857964883754e-10\\
131072	3.29752819722234e-11\\
};
\addlegendentry{$\calO(N^{-1.98})$};
		\end{axis}
		\end{tikzpicture}
		\caption{Weight sequence $\bsgamma=(\gamma_j)_{j=1}^d$ with $\gamma_j = (0.3)^j$.}
	\end{subfigure}
	\begin{subfigure}[b]{0.5\textwidth}  
		\centering 
		\begin{tikzpicture}
		
	\begin{axis}[%
width=0.811\textwidth,
height=0.8\textwidth,
at={(0\textwidth,0\textwidth)},
scale only axis,
xmode=log,
xmin=10,
xmax=850000,
xminorticks=true,
xlabel style={font=\color{white!15!black}},
xlabel={Number of points $N=2^m$},
ymode=log,
ymin=1000,
ymax=10000000,
yminorticks=true,
ylabel style={font=\color{white!15!black}},
ylabel={Worst-case error $e_{N,d,\alpha,\mathbf{\gamma^{\alpha}}}(\mathbf{z})$},
axis background/.style={fill=white},
xmajorgrids,
ymajorgrids,
legend style={at={(0.03,0.03)}, anchor=south west, legend cell align=left, align=left, draw=white!15!black}
]
		\addplot [color=mycolor1-fig, line width=0.6pt, mark=o, mark options={solid}, forget plot]
  table[row sep=crcr]{%
64	4417079.48643413\\
128	2208539.31795966\\
256	1104269.19820593\\
512	552134.113436488\\
1024	276066.703147798\\
2048	138033.261541012\\
4096	69016.3247481713\\
8192	34508.0564800858\\
16384	17253.8180482969\\
32768	8626.7707941766\\
65536	4313.28896327669\\
131072	2156.56910773271\\
};

		\addplot [color=mycolor2-fig,solid,line width=0.6pt,mark=triangle,mark options={solid},forget plot]
		table[row sep=crcr]{%
			64	4442001.0870858\\
128	2787076.95406399\\
256	1281324.27199222\\
512	947953.582658419\\
1024	716646.703824302\\
2048	195917.313724991\\
4096	103624.865472284\\
8192	59354.6025559344\\
16384	22436.2033938627\\
32768	9878.71828265001\\
65536	6681.40887658143\\
131072	4024.26849863592\\
};

		\addplot [color=mycolor3-fig,dashed,line width=0.8pt]
		table[row sep=crcr]{%
		64	2443100.59789719\\
128	1239301.9159449\\
256	628655.750068841\\
512	318895.7000791\\
1024	161764.952468506\\
2048	82057.8635605536\\
4096	41625.165830858\\
8192	21115.0321890582\\
16384	10710.9383337145\\
32768	5433.2950554575\\
65536	2756.12595646615\\
131072	1398.08904364146\\
};
\addlegendentry{$\calO(N^{-0.98})$}


		\end{axis}
		\end{tikzpicture}
		\caption{Weight sequence $\bsgamma=(\gamma_j)_{j=1}^d$ with $\gamma_j = (0.95)^j$.}
	\end{subfigure}
	\vskip\baselineskip
	\begin{tikzpicture}
	\hspace{0.05\linewidth}
	\begin{customlegend}[
	legend columns=5,legend style={align=left,draw=none,column sep=1.5ex},
	legend entries={CBC-DBD, reduced CBC-DBD \quad, $\alpha=2$}
	]
	\addlegendimage{color=mycolor1-fig, mark=o,solid,line width=1.0pt,line legend}
	\addlegendimage{color=mycolor2-fig, mark=triangle,solid,line width=1.0pt}  
	\end{customlegend}
	\end{tikzpicture}
	\caption{Convergence of the worst-case errors $e_{N,d,\alpha,\bsgamma^{\alpha}}(\bsz)$ in the weighted space $E_{d,\bsgamma}^\alpha$ 
	for smoothness parameter $\alpha=2$ with dimension $d=100$ and reduction indices $w_j = \lfloor 2\log_2 j\rfloor$. The generating vector $\bsz$ is constructed via the reduced 
	CBC-DBD construction and the non-reduced CBC-DBD construction for $N=2^m$, respectively.}  
	\label{fig:red_cbc_dbd_1}
\end{figure}
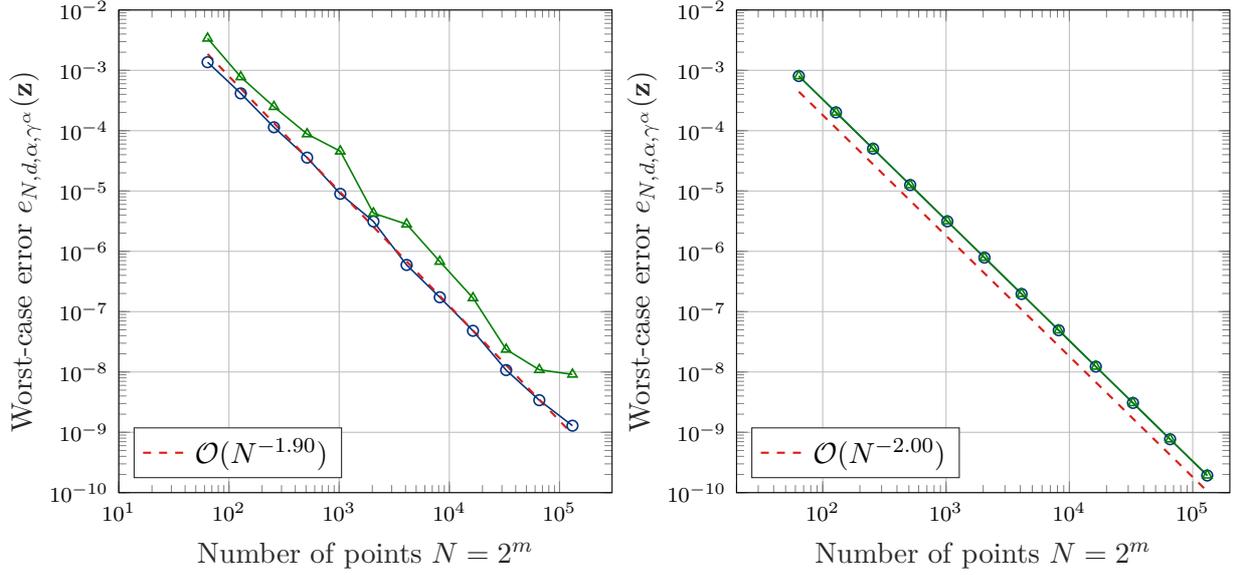
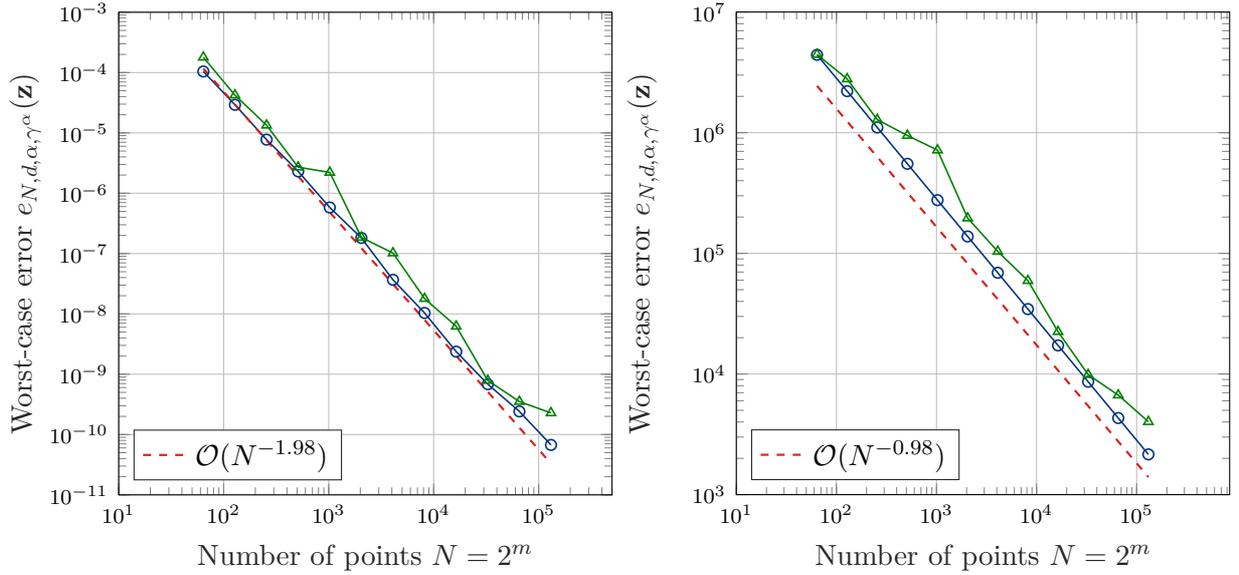

\begin{figure}
	\centering
	\textbf{Error convergence in the space $E_{d,\bsgamma}^{\alpha}$ with $d=100, \alpha=2$, $w_j = \lfloor\frac{7}{2}\log_2 j\rfloor$}. \par\medskip 
	\hspace{-0.25cm}
	\centering
	\begin{subfigure}[b]{0.5\textwidth}
		\centering
		\begin{tikzpicture}
		\begin{axis}[%
		width=0.8\textwidth,
		height=0.8\textwidth,
		at={(0\textwidth,0\textwidth)},
		scale only axis,
		xmode=log,
		xmin=10,
		xmax=200000,
		xminorticks=true,
		xlabel={Number of points $N=2^m$},
		xmajorgrids,
		ymode=log,
		ymin=1e-11,
		ymax=0.1,
		yminorticks=true,
		ylabel={Worst-case error $e_{N,d,\alpha,\mathbf{\gamma^{\alpha}}}(\bsz)$},
		ymajorgrids,
		axis background/.style={fill=white},
		legend style={at={(0.03,0.03)},anchor=south west,legend cell align=left,align=left,draw=white!15!black}
		]
		\addplot [color=mycolor1-fig,solid, line width=0.6pt, mark=o, mark options={solid}, forget plot]
  table[row sep=crcr]{%
64	0.00136397342871476\\
128	0.00041495206968083\\
256	0.000113777898495654\\
512	3.57186608991275e-05\\
1024	8.96735413484793e-06\\
2048	3.12393880750252e-06\\
4096	5.94557183683682e-07\\
8192	1.73173640260874e-07\\
16384	4.81706937364865e-08\\
32768	1.07434643818022e-08\\
65536	3.40729981226737e-09\\
131072	1.28708617753193e-09\\
};
		\addplot [color=mycolor2-fig,solid,line width=0.6pt,mark=triangle,mark options={solid},forget plot]
		table[row sep=crcr]{%
	64	0.00438960243064497\\
128	0.00157546948831783\\
256	0.000823138390056791\\
512	0.000721210426734021\\
1024	0.000260416773260162\\
2048	0.000129161777169\\
4096	4.81379702335332e-05\\
8192	3.43147138966629e-05\\
16384	4.50251027764988e-06\\
32768	1.27098997245945e-06\\
65536	1.03687431640195e-06\\
131072	3.47127881156422e-07\\
};
		\addplot [color=mycolor3-fig,dashed,line width=0.8pt]
		table[row sep=crcr]{%
			64	0.00241428133685473\\
128	0.000976439443559482\\
256	0.000394914201747863\\
512	0.000159720326509579\\
1024	6.45977850060049e-05\\
2048	2.61261288332752e-05\\
4096	1.05665327030863e-05\\
8192	4.2735613101314e-06\\
16384	1.72841241158678e-06\\
32768	6.99044484852697e-07\\
65536	2.8272372295357e-07\\
131072	1.14345660759444e-07\\
};
\addlegendentry{$\calO(N^{-1.31})$};
		
		\end{axis}
		\end{tikzpicture}
		\caption{Weight sequence $\bsgamma=(\gamma_j)_{j\ge 1}$ with $\gamma_j = 1/j^3$.}    
	\end{subfigure}
	\begin{subfigure}[b]{0.5\textwidth}  
		\centering 
		\begin{tikzpicture}
		
		\begin{axis}[%
		width=0.8\textwidth,
		height=0.8\textwidth,
		at={(0\textwidth,0\textwidth)},
		scale only axis,
		xmode=log,
		xmin=10,
		xmax=500000,
		xminorticks=true,
		xlabel={Number of points $N=2^m$},
		xmajorgrids,
		ymode=log,
		ymin=1e-10,
		ymax=0.01,
		yminorticks=true,
		ylabel={Worst-case error $e_{N,d,\alpha,\mathbf{\gamma^{\alpha}}}(\bsz)$},
		ymajorgrids,
		axis background/.style={fill=white},
		legend style={at={(0.03,0.03)},anchor=south west,legend cell align=left,align=left,draw=white!15!black}
		]
		
		\addplot [color=mycolor1-fig,solid, line width=0.6pt, mark=o, mark options={solid}, forget plot]
  table[row sep=crcr]{%
64	0.000803660679099214\\
128	0.000200971272724395\\
256	5.02498739194294e-05\\
512	1.25700964904643e-05\\
1024	3.14241304450872e-06\\
2048	7.86290904231181e-07\\
4096	1.96387468195797e-07\\
8192	4.9114858240217e-08\\
16384	1.22747998201088e-08\\
32768	3.06929454004206e-09\\
65536	7.68023623822221e-10\\
131072	1.92351367445007e-10\\
};
\addplot [color=mycolor2-fig,solid,line width=0.6pt,mark=triangle,mark options={solid},forget plot]
		table[row sep=crcr]{%
	64	0.000804859087734162\\
128	0.000201129469271601\\
256	5.04217904285333e-05\\
512	1.25760193998465e-05\\
1024	3.15927419786685e-06\\
2048	8.04890774905338e-07\\
4096	1.96638108254771e-07\\
8192	4.92394329609272e-08\\
16384	1.23455284932492e-08\\
32768	3.11648815713363e-09\\
65536	7.68827731927967e-10\\
131072	1.9377823595292e-10\\
};		
		\addplot [color=mycolor3-fig,dashed,line width=0.8pt]
		table[row sep=crcr]{%
		64	0.000442672498253789\\
128	0.000110851906553729\\
256	2.77589984357955e-05\\
512	6.9512741649149e-06\\
1024	1.74070446480893e-06\\
2048	4.35898795230852e-07\\
4096	1.09155668595682e-07\\
8192	2.73342347281787e-08\\
16384	6.84490689111794e-09\\
32768	1.71406848642349e-09\\
65536	4.29228742316797e-10\\
131072	1.07485386196722e-10\\
};
\addlegendentry{$\calO(N^{-2.00})$}
		\end{axis}
		\end{tikzpicture}
		\caption{Weight sequence $\bsgamma=(\gamma_j)_{j\ge 1}$ with $\gamma_j = 1/j^8$.}
	\end{subfigure}
	\vskip\baselineskip
	\hspace{-0.25cm}
	\centering
	\begin{subfigure}[b]{0.5\textwidth}
		\centering
		\begin{tikzpicture}
		
		\begin{axis}[%
		width=0.8\textwidth,
		height=0.8\textwidth,
		at={(0\textwidth,0\textwidth)},
		scale only axis,
		xmode=log,
		xmin=10,
		xmax=1000000,
		xminorticks=true,
		xlabel={Number of points $N=2^m$},
		xmajorgrids,
		ymode=log,
		ymin=1e-12,
		ymax=0.01,
		yminorticks=true,
		ylabel={Worst-case error $e_{N,s,\alpha,\mathbf{\gamma^{\alpha}}}(\bsz)$},
		ymajorgrids,
		axis background/.style={fill=white},
		legend style={at={(0.03,0.03)},anchor=south west,legend cell align=left,align=left,draw=white!15!black}
		]

	\addplot [color=mycolor1-fig, line width=0.6pt, mark=o, mark options={solid}, forget plot]
  table[row sep=crcr]{%
64	0.000104240109217907\\
128	2.9195489069013e-05\\
256	7.72906106593819e-06\\
512	2.30081487959806e-06\\
1024	5.79809038743061e-07\\
2048	1.81503073600323e-07\\
4096	3.66692191473153e-08\\
8192	1.03498552821969e-08\\
16384	2.3686872108615e-09\\
32768	6.82853999181353e-10\\
65536	2.41399001025539e-10\\
131072	6.72599684960036e-11\\
};
	
		\addplot [color=mycolor2-fig,solid,line width=0.6pt,mark=triangle,mark options={solid},forget plot]
		table[row sep=crcr]{%
	64	0.00077137433981393\\
128	0.000199434230406031\\
256	7.65663229654456e-05\\
512	2.10826176438758e-05\\
1024	3.3446608337113e-06\\
2048	1.87176435666927e-06\\
4096	6.54761301410325e-07\\
8192	5.48682628528749e-07\\
16384	2.11409663781009e-08\\
32768	1.12341188491756e-08\\
65536	1.04699787283312e-08\\
131072	1.84349979762201e-10\\
};
	
		\addplot [color=mycolor3-fig,dashed,line width=0.8pt]
		table[row sep=crcr]{%
	64	0.000501393320879055\\
128	0.000145467014305223\\
256	4.22036979147961e-05\\
512	1.22443711805763e-05\\
1024	3.55240495537636e-06\\
2048	1.03064345084551e-06\\
4096	2.99016000741448e-07\\
8192	8.67521824604423e-08\\
16384	2.51690248782284e-08\\
32768	7.30217725196412e-09\\
65536	2.11854821063117e-09\\
131072	6.14644970383496e-10\\
};
\addlegendentry{$\calO(N^{-1.79})$}
		\end{axis}
		\end{tikzpicture}
		\caption{Weight sequence $\bsgamma=(\gamma_j)_{j\ge 1}$ with $\gamma_j = (0.3)^j$.}
	\end{subfigure}
	\begin{subfigure}[b]{0.5\textwidth}  
		\centering 
		\begin{tikzpicture}
		
		\begin{axis}[%
		width=0.8\textwidth,
		height=0.8\textwidth,
		at={(0\textwidth,0\textwidth)},
		scale only axis,
		xmode=log,
		xmin=10,
		xmax=1000000,
		xminorticks=true,
		xlabel={Number of points $N=2^m$},
		xmajorgrids,
		ymode=log,
		ymin=1000,
        ymax=10e7,
		yminorticks=true,
		ylabel={Worst-case error $e_{N,d,\alpha,\mathbf{\gamma^{\alpha}}}(\bsz)$},
		ymajorgrids,
		axis background/.style={fill=white},
		legend style={at={(0.03,0.03)},anchor=south west,legend cell align=left,align=left,draw=white!15!black}
		]
		\addplot [color=mycolor1-fig, line width=0.6pt, mark=o, mark options={solid}, forget plot]
  table[row sep=crcr]{%
64	4417079.48643413\\
128	2208539.31795966\\
256	1104269.19820593\\
512	552134.113436488\\
1024	276066.703147798\\
2048	138033.261541012\\
4096	69016.3247481713\\
8192	34508.0564800858\\
16384	17253.8180482969\\
32768	8626.7707941766\\
65536	4313.28896327669\\
131072	2156.56910773271\\
};
		\addplot [color=mycolor2-fig,solid,line width=0.6pt,mark=triangle,mark options={solid},forget plot]
		table[row sep=crcr]{%
	64	4661989.50380705\\
128	2489801.94909981\\
256	2023402.86809754\\
512	845692.60196366\\
1024	900090.312317118\\
2048	794505.468702513\\
4096	559608.308745691\\
8192	151033.87256604\\
16384	103683.082594485\\
32768	25133.0476211002\\
65536	14791.3113042966\\
131072	11758.2069958603\\
};
	
		\addplot [color=mycolor3-fig,dashed,line width=0.8pt]
		table[row sep=crcr]{%
	64	2564094.22709388\\
128	1435151.48071704\\
256	803269.923094327\\
512	449598.929463241\\
1024	251645.420254028\\
2048	140848.683982473\\
4096	78834.5433013182\\
8192	44124.5529727547\\
16384	24696.9931392103\\
32768	13823.1761916051\\
65536	7736.98235032462\\
131072	4330.47333402207\\
};
\addlegendentry{$\calO(N^{-0.84})$}

		\end{axis}
		\end{tikzpicture}
		\caption{Weight sequence $\bsgamma=(\gamma_j)_{j\ge 1}$ with $\gamma_j = (0.95)^j$.}
	\end{subfigure}
	\vskip\baselineskip
	\begin{tikzpicture}
	\hspace{0.05\linewidth}
	\begin{customlegend}[
	legend columns=5,legend style={align=left,draw=none,column sep=1.5ex},
	legend entries={CBC-DBD, reduced CBC-DBD \quad, $\alpha=2$, $\alpha=3$, $\alpha=4$}
	]
	\addlegendimage{color=mycolor1-fig, mark=o,solid,line width=1.0pt,line legend}
	\addlegendimage{color=mycolor2-fig, mark=triangle,solid,line width=1.0pt}  
	\end{customlegend}
	\end{tikzpicture}
	\caption{Convergence of the worst-case errors $e_{N,d,\alpha,\bsgamma^{\alpha}}(\bsz)$ in the weighted space $E_{d,\bsgamma}^\alpha$ 
	for smoothness parameter $\alpha=2$ with dimension $d=100$ and reduction indices $w_j = \lfloor \frac{7}{2}\log_2 j\rfloor$. The generating vector $\bsz$ is constructed via the reduced 
	CBC-DBD construction and the non-reduced CBC-DBD construction for $N=2^m$, respectively.}  
	\label{fig:red_cbc_dbd_2}
\end{figure}
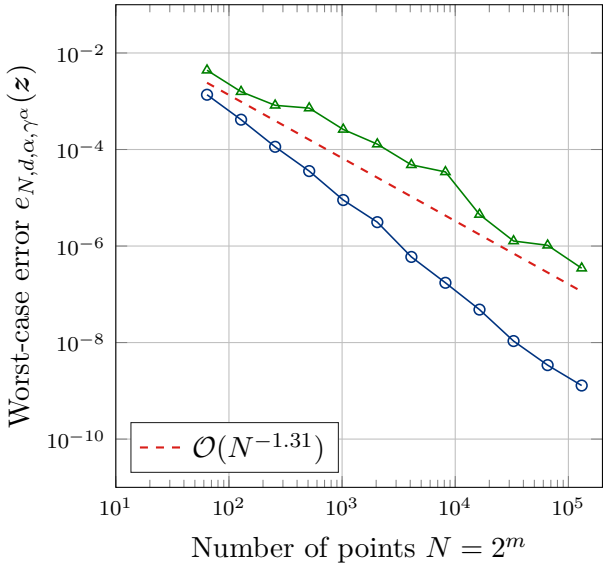
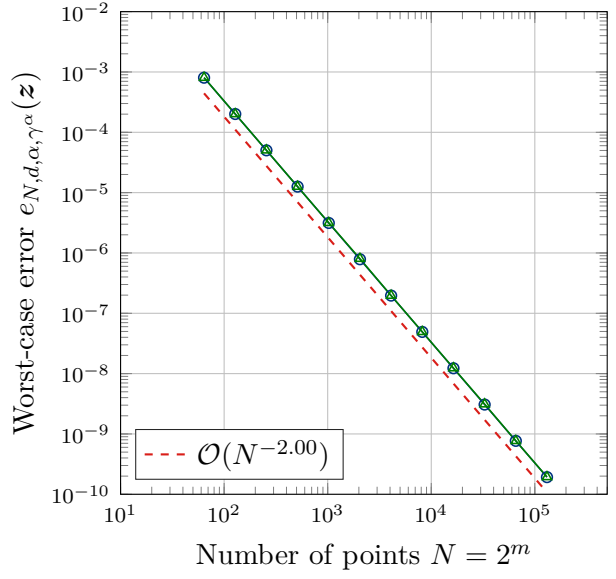
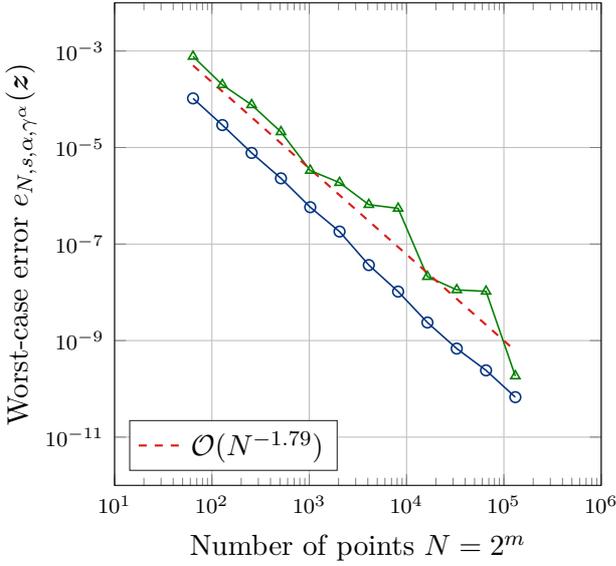
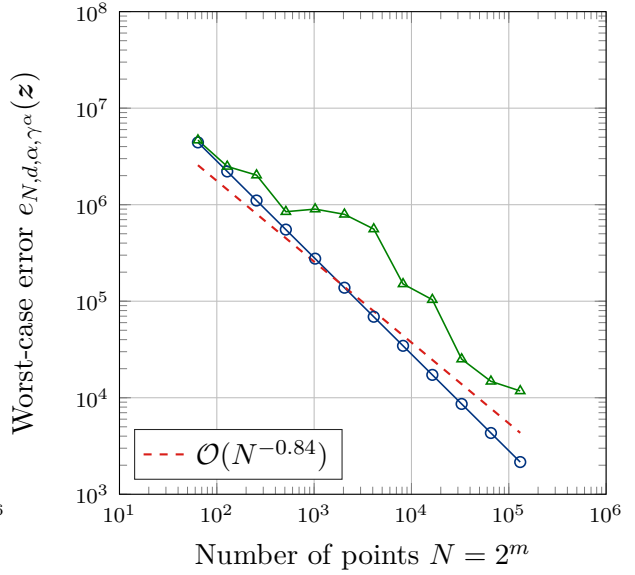

\newpage

The results in Figures \ref{fig:red_cbc_dbd_1} and \ref{fig:red_cbc_dbd_2} show that the reduced fast CBC-DBD algorithm constructs generating vectors of good lattice rules which have worst-case errors that are essentially comparable to those of lattice rules obtained by the non-reduced fast CBC-DBD algorithm. We observe asymptotic error rates for both algorithms considered. Only in Part (a) of Figures \ref{fig:red_cbc_dbd_1} and \ref{fig:red_cbc_dbd_2}, respectively, the errors of the lattice rules obtained 
by the reduced construction seem to be significantly higher than those of the non-reduced case. This error behavior can be explained as follows. As illustrated in Figures \ref{fig:red_cbc_dbd_1} and \ref{fig:red_cbc_dbd_2}, throughout we use $b= 2$ and $\alpha = 2$. Corollary \ref{cor:main-result-red_dbd} assures independence of 
the dimension $d$ whenever the chosen weights $\gamma_j$ satisfy 
\begin{equation}\label{eqn_cor}
\sum_{j\geq 1} \gamma_j b^{w_j} = \sum_{j\geq 1} \gamma_j 2^{w_j} < \infty.
\end{equation}
However, the sequence of weights $\gamma_j = 1/j^3$ does not always satisfy Condition \eqref{eqn_cor} if we 
choose the reduction indices as $w_j=\lfloor p \log_2 j \rfloor$, since
\[
\sum_{j\geq1} \gamma_j 2^{w_j} = \sum_{j\geq 1} j^{-3} 2^{ \lfloor p \log_2 j\rfloor} \ge  \sum_{j\geq 1} j^{-3} 2^{ p \log_2 j - 1} \ge 
\frac{1}{2}\sum_{j\geq 1} j^{p-3}.
\]
The latter series only convergence if whenever $3-p > 1$, and this is not
satisfied for our choices of $p$ made in Figures \ref{fig:red_cbc_dbd_1} and \ref{fig:red_cbc_dbd_2}, 
which is $p=2$ and $p=7/2$, respectively. Thus, this gives rise to the difference in the errors obtained by the reduced CBC-DBD and non-reduced CBC-DBD algorithms.
This phenomenon is to be expected and shows that the reduction indices and the weights must be balanced carefully, and that in general the reduced algorithm 
works better for situations where the weights $\gamma_j$ decay fast.

\subsection{Computational complexity}
Here, we illustrate the computational complexity of the reduced fast CBC-DBD construction in Algorithm \ref{alg:fast-redcbcdbd} 
which was proved in Theorem~\ref{thm:cost-redcbcdbd}. For this purpose, let $N=2^m$, let the weight sequence $\bsgamma = (\gamma_j)_{j\geq 1}$ be given by 
$\gamma_j = (0.95)^j$, and let the reduction indices be given by $w_j = \lfloor \frac{3}{2}\log_2j\rfloor$ for $j\ge 1$. We measure and compare the computation times of implementations of Algorithm~\ref{alg:fast-redcbcdbd} and the non-reduced fast CBC-DBD algorithm for lattice rules 
(see \cite{EKNO21} for details on the implementation of the latter). Note that the chosen weight sequence does not affect the computation times. The timings were performed on an Intel Core i5 CPU with 2.3 GHz using Python 3.6.3.
\begin{table}[H]
	\captionof{table}{Computation times (in seconds) for constructing the generating vector $\bsz$ of a lattice rule with $N=2^m$ points in $d$ dimensions using the 
	reduced CBC-DBD algorithm (\textbf{bold font}) and the non-reduced CBC-DBD construction (normal font) with $\alpha = 2$, $\gamma_j = (0.95)^j$, and $w_j = \lfloor \frac{3}{2} \log_2j\rfloor$.} 
	\label{tab:red_dbd_times}
	\centering
\begin{tabular}{p{2cm}p{2cm}p{2cm}p{2cm}p{2cm}p{2cm}} 
\toprule[1.2pt]
& $d=50$ & $d=100$ & $d=500$ & $d=1000$ & $d=2000$ \\ 
\toprule[1.2pt]
\multirow{2}{6em}{$m=10$}
& 0.077 & 0.154 & 0.799 & 1.571 & 3.182 \\
& \textbf{0.017} & \textbf{0.018} & \textbf{0.02} & \textbf{0.021} & \textbf{0.025} \\
\midrule
\multirow{2}{6em}{$m=12$}
& 0.128 & 0.252 & 1.224 & 2.424 & 4.908 \\
& \textbf{0.035} & \textbf{0.046} & \textbf{0.052} & \textbf{0.054} & \textbf{0.057} \\
\midrule
\multirow{2}{6em}{$m=14$}
& 0.211 & 0.415 & 2.044 & 4.049 & 8.256 \\
& \textbf{0.066} & \textbf{0.089} & \textbf{0.138} & \textbf{0.136} & \textbf{0.141} \\
\midrule
\multirow{2}{6em}{$m=16$}
& 0.43 & 0.874 & 4.363 & 8.796 & 17.631 \\
& \textbf{0.103} & \textbf{0.152} & \textbf{0.299} & \textbf{0.354} & \textbf{0.359} \\
\midrule
\multirow{2}{6em}{$m=18$}
& 1.467 & 2.982 & 14.924 & 30.545 & 59.967 \\
& \textbf{0.195} & \textbf{0.272} & \textbf{0.577} & \textbf{0.761} & \textbf{0.904} \\
\midrule
\multirow{2}{6em}{$m=20$}
& 7.222 & 14.538 & 73.21 & 147.759 & 294.616 \\
& \textbf{0.5} & \textbf{0.623} & \textbf{1.146} & \textbf{1.516} & \textbf{1.931} \\
\midrule
\end{tabular}
\end{table}
In Table \ref{tab:red_dbd_times} we display the computation times for the construction of the generating vector $\bsz$ via the two considered algorithms, where for the reduction indices we use $w_j = \lfloor \frac{3}{2}\log_2 j \rfloor$. We emphasize that the used algorithms solely construct the generating vector $\bsz$ but do not calculate the worst-case error $e_{N,d,\alpha,\bsgamma}(\bsz)$, which allows for an unbiased comparison between the considered algorithms. Table \ref{tab:red_dbd_times} illustrates a dramatic difference in the computational cost between the non-reduced fast CBC-DBD construction and the reduced fast CBC-DBD construction. The extent of the speed-up depends on the chosen reduction indices $w_j$. Note, however, that the reduction indices have 
to be chosen such that they are balanced with the weights, in order to guarantee useful error convergence.

\section{Conclusion}
In this paper, we have presented a combination of the CBC-DBD algorithm introduced in \cite{EKNO21} and the reduced construction method in \cite{DKLP15} for constructing good lattice rules for numerical integration in weighted Korobov spaces. In particular, we have aimed to gain from the reduced construction method to shrink the computational cost as compared to the non-reduced CBC-DBD algorithm. We showed that the reduced CBC-DBD construction with quality measure independent of the smoothness parameter $\alpha$, similarly to \cite{EKNO21} also formulated for product weights, yields lattice rules which admit error convergence rates that are arbitrarily close to the optimal convergence order. 
We remark that there has recently been considerable interest in finding algorithms that guarantee some degree of universality with respect to the smoothness parameter 
and/or the weights (see, e.g., \cite{DG21}), and also our result can be seen as a step in this direction. 
Furthermore, the errors can be bounded independently of the dimension if the weights satisfy suitable summability conditions. In addition to these theoretical results, we have derived a fast implementation of the considered algorithm. Numerical tests confirm our main findings.

\bigskip

\begin{small}
	\noindent\textbf{Authors' addresses:}\\
	
	\noindent Peter Kritzer\\
	Johann Radon Institute for Computational and Applied Mathematics (RICAM)\\
	Austrian Academy of Sciences\\
	Altenbergerstr. 69, 4040 Linz, Austria.\\
	\texttt{peter.kritzer@oeaw.ac.at}\\
	
	\noindent Onyekachi Osisiogu\\
	Department of Applied Mathematics\\
	Illinois Institute of Technology\\
	Chicago, IL USA.\\
    \texttt{oosisiogu@iit.edu}\\
    and\\
    Johann Radon Institute for Computational and Applied Mathematics (RICAM)\\
	Austrian Academy of Sciences\\
	Altenbergerstr. 69, 4040 Linz, Austria.\\
\end{small}


\bigskip


\begin{thebibliography}{99}
	
	
	\bibitem{CKNS16} R.~Cools, F.Y.~Kuo, D.~Nuyens, G.~Suryanarayana. 
	Tent-transformed lattice rules for integration and approximation of multivariate non-periodic functions. 
	J. Complexity 36, 166--181, 2016. 
	
	\bibitem{D04} J.~Dick. 
	On the convergence rate of the component-by-component construction of good lattice rules. 
	J. Complexity 20, 493--522, 2004. 
	
	\bibitem{DG21}
	J.~Dick, T.~Goda.
	Stability of lattice rules and polynomial lattice rules constructed by the component-by-component algorithm.
	J. Comput. Appl. Math. 382, 113062, 2021.
	
	\bibitem{DKLP15} J.~Dick, P.~Kritzer, G.~Leobacher, F.~Pillichshammer. 
	A reduced fast component-by-component construction of lattice points for integration in weighted spaces with fast decreasing weights. 
	J. Comput. Appl. Math. 276, 1--15, 2015.	
	
	\bibitem{DKP22} J.~Dick, P.~Kritzer, F.~Pillichshammer. 
	{\it Lattice Rules}. 
	Springer, Cham, 2022.
	
	\bibitem{DKS13} J.~Dick, F.Y.~Kuo, I.H.~Sloan.
	High-dimensional integration---the quasi-Monte Carlo way. 
	Acta Numer. 22, 133--288, 2013.
	
	\bibitem{DNP14} J.~Dick, D.~Nuyens, F.~Pillichshammer. 
	Lattice rules for nonperiodic smooth integrands. 
	Numer. Math. 126, 259--291, 2014.
	
	\bibitem{EKNO21} A.~Ebert, P.~Kritzer, D.~Nuyens, O. ~Osisiogu.
	Digit-by-digit and component-by-component constructions of lattice rules for periodic functions with unknown smoothness.
	Journal of Complexity, 66, 101555, 2021.
	
	\bibitem{GSY19} T.~Goda, K.~Suzuki, T.~Yoshiki.
	Lattice rules in non-periodic subspaces of Sobolev spaces.
	Numer. Math. 141, 399--427, 2019.
	
	\bibitem{Hic98} F.J.~Hickernell. 
	A generalized discrepancy and quadrature error bound. 
	Math. Comp., 67, 299--322, 1998.
	
	\bibitem{HN03} F.J.~Hickernell, H.~Niederreiter. 
	The existence of good extensible rank-$1$ lattices. 
	J. Complexity, 19, 286--300, 2003.
	
	\bibitem{H62} E.~Hlawka.
	Zur angen\"{a}herten Berechnung mehrfacher Integrale.
	Monatsh. Math., 66, 140--151, 1962.
	
	\bibitem{HW81} L.K.~Hua, Y.~Wang.
	{\it Applications of Number Theory to Numerical Analysis}.
	Springer, Berlin, 1981. 
	
	\bibitem{Kor59} N.M.~Korobov. 
	Approximate evaluation of repeated integrals. 
	Dokl. Akad. Nauk SSSR, 124, 1207--1210, 1959. In Russian. English translation of the theorems in \cite{S71}.
	
	\bibitem{Kor63} N.M.~Korobov. 
	{\it Number-theoretic methods in approximate analysis}. 
	Goz. Izdat. Fiz.-Math., 1963. In Russian. English translation of results on optimal coefficients in \cite{S71}.
	
	\bibitem{Kor82} N.M.~Korobov. 
	{\it On the computation of optimal coefficients}. 
	Dokl. Akad. Nauk SSSR, 267:289--292, 1982. In Russian. English translation see \cite{Kor82Eng}.
	
	\bibitem{Kor82Eng} N.M.~Korobov. 
	On the computation of optimal coefficients. 
	Dokl. Akad. Nauk SSSR, 26:590--593, 1982.
	
	\bibitem{K03} F.Y.~Kuo. 
	Component-by-component constructions achieve the optimal rate of convergence for multivariate integration in weighted Korobov and Sobolev spaces. 
	J. Complexity 19, 301--320, 2003.
	
	\bibitem{N92b} H.~Niederreiter. 
	{\it Random Number Generation and Quasi-Monte Carlo Methods}. 
	SIAM, Philadelphia, 1992. 
	
	\bibitem{NW08} E.~Novak, H.~Wo\'{z}niakowski. 
	{\it Tractability of Multivariate Problems. Volume I: Linear Information}.
	EMS, Zurich, 2008.
	
	\bibitem{NC06b} D.~Nuyens, R.~Cools. 
	Fast algorithms for component-by-component construction of rank-$1$ lattice rules in 
    shift-invariant reproducing kernel Hilbert spaces. 
	Math. Comp. 75, 903--920, 2006.
	
	\bibitem{SJ94} I.H.~Sloan, S.~Joe.
	{\it Lattice Methods for Multiple Integration}.
	Clarendon Press, Oxford, 1994.
	
	\bibitem{SW98} I.H.~Sloan, H.~Wo\'zniakowski. 
	When are quasi Monte Carlo algorithms efficient for high-dimensional problems?. 
	J. Complexity 14, 1--33, 1998. 
	
	\bibitem{SW01} I.H.~Sloan, H.~Wo\'zniakowski. 
	Tractability of multivariate integration for weighted Korobov classes. 
	J. Complexity 17, 697--721, 2001. 
	
	\bibitem{S71} A.H.~Stroud. 
	{\it Approximate Calculation of Multiple Integrals}. 
	Prentice-Hall Series in Automatic Computation, 1971.
\end{thebibliography}
\end{document}